\documentclass[letterpaper,11pt,twoside,keywordsasfootnote,addressatend,noinfoline]{article}
\usepackage{fullpage}
\usepackage[english]{babel}
\usepackage{amssymb}
\usepackage{amsmath}
\usepackage{theorem}
\usepackage{epsfig}
\usepackage{imsart}
\usepackage{color}

\theorembodyfont{\slshape}
\newtheorem{theorem}{Theorem}
\newtheorem{definition}[theorem]{Definition}

\newtheorem{proposition}[theorem]{Proposition}
\newtheorem{lemma}[theorem]{Lemma}
\newtheorem{corollary}[theorem]{Corollary}

\newcommand{\N}{\mathbb{N}}
\newcommand{\Z}{\mathbb{Z}}
\newcommand{\R}{\mathbb{R}}
\renewcommand{\u}{\mathbf{u}}
\renewcommand{\v}{\mathbf{v}}
\newcommand{\w}{\mathbf{w}}
\newcommand{\x}{\mathbf{x}}
\newcommand{\y}{\mathbf{y}}
\newcommand{\z}{\mathbf{z}}
\newcommand{\0}{\mathbf{0}}
\newcommand{\ep}{\epsilon}

\DeclareMathOperator{\uniform}{Uniform \,}
\DeclareMathOperator{\poisson}{Poisson \,}
\DeclareMathOperator{\card}{card \,}
\DeclareMathOperator{\var}{Var}

\newenvironment{proof}{\noindent{\scshape Proof.}}{\hspace{2mm} $\square$}
\newenvironment{demo}[1]{\noindent{\textbf{Proof of #1}}}{\hspace*{2mm}~$\square$}

\begin{document}

\begin{frontmatter}

\title     {Survival and extinction results for a patch \\ model with sexual reproduction}
\runtitle  {Survival and extinction results for a patch model with sexual reproduction}
\author    {Eric Foxall\thanks{Research supported in part by an NSERC PGS D2 Award} and Nicolas Lanchier\thanks{Research supported in part by NSF Grant DMS-10-05282 and NSA Grant MPS-14-040958.}}
\runauthor {E. Foxall and N. Lanchier}
\address   {Department of Mathematics, \\ University of Victoria, \\ Victoria, BC V8X 1P3, Canada.}
\address   {School of Mathematical \\ and Statistical Sciences, \\ Arizona State University, \\ Tempe, AZ 85287, USA.}

\begin{abstract} \ \
 This article is concerned with a version of the contact process with sexual reproduction on a graph with two levels of interactions modeling metapopulations.
 The population is spatially distributed into patches and offspring are produced in each patch at a rate proportional to the number of pairs of individuals in the patch (sexual reproduction) rather than simply the number of individuals
 as in the basic contact process.
 Offspring produced at a given patch either stay in their parents' patch or are sent to a nearby patch with some fixed probabilities.
 As the patch size tends to infinity, we identify a mean-field limit consisting of an infinite set of coupled differential equations.
 For the mean-field equations, we find explicit conditions for survival and extinction that we call expansion and retreat.
 Using duality techniques to compare the stochastic model to its mean-field limit, we find that expansion and retreat are also precisely the conditions needed to ensure survival and extinction of the stochastic model when the patch size is large.
 In addition, we study the dependence of survival on the dispersal range.
 We find that, with probability close to one and for a certain set of parameters, the metapopulation survives in the presence of nearest neighbor interactions while it dies out in the presence of long range interactions,
 suggesting that the best strategy for the population to spread in space is to use intermediate dispersal ranges.
\end{abstract}



\end{frontmatter}


\section{Introduction}
\label{sec:intro}

\indent The term Allee effect refers to a certain process that leads to decreasing net population growth with decreasing density \cite{allee_1931}.
 In case the growth rate becomes negative at low density, this monotone relationship results in the existence of a so-called Allee threshold below which populations are at high risk of being driven toward extinction.
 This phenomenon may be due to various ecological factors: failure to locate mates, inbreeding depression, failure to satiate predators, lack of cooperative feeding, etc.
 Research on this topic is copious and is reviewed in~\cite{courchamp_berec_gascoigne_2009} but rigorous mathematical analyses of stochastic spatial models that include an Allee effect are much more limited.

\indent In the model proposed by Borrello~\cite{borrello_2012}, each site of the infinite regular lattice represents a patch that can host a local population, and a strong Allee effect is included in the form of a
 varying individual death rate taking a larger value for local populations below some threshold.
 The model is used to show that, when only small flocks of individuals can migrate from patch to patch, the metapopulation goes extinct whereas survival is possible if large enough flocks of individuals can migrate.
 The framework of interacting particle systems has also been used in~\cite{kang_lanchier_2011,lanchier_2013} to study the consequence of an Allee effect.
 There, the model is a modification of the averaging process that also includes a threshold: local populations below this threshold go extinct whereas local populations above this threshold expand to their carrying
 capacity, each at rate one.
 The key component in this work is the topology of the network of interactions rather than the size of the migrating flocks, and the analysis of the process starting from a single occupied patch on various graphs
 indicates that the probability of long-term survival of the metapopulation decreases to zero as the degree of the network of interactions increases to infinity.
 This result suggests that long range dispersal promotes extinction of metapopulations subject to a strong Allee effect.

\indent The modeling approach of the present paper is somewhat different.
 The model we propose is a version of the contact process with sexual reproduction~\cite{neuhauser_1994, noble_1992} on a graph that includes two levels of interactions modeling metapopulations:
 individuals are produced within each patch at a rate proportional to the number of pairs of individuals in the patch rather than a rate proportional to the number of individuals in the patch.
 This birth mechanism (sexual reproduction) reflects the difficulty to locate mates in patches at low density, which has been identified by ecologists as one of the most common causes of Allee effect.
 In particular, while the Allee effect is forced into the model in the form of a threshold parameter in~\cite{borrello_2012,kang_lanchier_2011,lanchier_2013}, it is on the contrary naturally induced by the birth mechanism in the
 model considered in this paper.
 A useful consequence of this fact is both our model and its mean-field equations have a dual process, moreover the stochastic model's dual process is a straightforward truncation of the mean-field's dual process.
 With this relationship in hand, we can easily estimate the occupation density of each patch in the stochastic process by the solutions to the  mean-field equations, and using block constructions, we can show the following.
 When a localized population spreads in the mean-field equations, we have survival of the stochastic process and existence of a non-trivial stationary distribution.
 When a localized low-density region erodes the surrounding population in the mean-field equations, the same occurs in the stochastic process and implies weak convergence to the vacant configuration.
 We analyze survival and extinction starting from either a large block of occupied patches, or a single occupied patch, and also assess the dependence of survival on the dispersal range.


\section{Model description and main results}
\label{sec:results}

\indent Each integer~$x \in \Z$ represents a patch that can host a local population of up to~$N$ individuals, and we think of each site as a set of spatial locations that can be either empty or occupied by one individual.
 By convention, we use bold letters to denote these spatial locations:
 $$ \x := (x, j) \in D_N := \Z \times \{1, 2, \ldots, N \} $$
 are the possible spatial locations at site/patch $x$.
 The model we consider is a continuous-time Markov chain whose state at time $t$ is a spatial configuration
 $$ \eta_t : D_N \longrightarrow \{0, 1 \} \quad \hbox{where} \quad 0 = \hbox{empty} \quad \hbox{and} \quad 1 = \hbox{occupied}. $$
 Each individual dies at rate one.
 Offspring produced at a given patch are sent to a spatial location chosen uniformly at random from either the parents' patch or a neighboring patch. 
 In either case, offspring are produced at a rate proportional to~$N$ times the fraction of pairs of spatial locations which are occupied and the birth is suppressed if the target is already occupied.
 The proportionality constant is denoted by~$a$ for offspring sent within their parents' patch and by~$b$ for offspring sent outside their parents' patch.
 Sexual reproduction is modeled by the fact that the birth rate is related to the number of occupied pairs and the reason for multiplying by~$N$ is to have births and deaths occurring at the same time scale.
 For simplicity, we consider symmetric interval neighborhoods but most of the basic facts, and some of the estimates, are true for any translation-invariant neighborhood, on $\Z$ or on more general transitive graphs.
 More formally, for $x,y \in \Z$ we write
 $$ x \sim y \quad \hbox{if and only if} \quad x \neq y \ \ \hbox{and} \ \ |x - y| \leq M $$
 where~$M$ is the dispersal range, and define the projection map
 $$ \pi \ : \ \x := (x, j) \in D_N \ \mapsto \ \pi (\x) := x \in \Z. $$
 For all~$\x \in D_N$, we also define the sets
 $$ \begin{array}{rcl}
     A (\x) & := & \hbox{set of potential parents' pairs within the patch containing $\x$} \vspace*{2pt} \\
            & := & \{(\y, \z) \in D_N \times D_N : \y \neq \z \ \hbox{and} \ \pi (\x) = \pi (\y) = \pi (\z) \} \vspace*{8pt} \\
     B (\x) & := & \hbox{set of potential parents' pairs near the patch containing $\x$} \vspace*{2pt} \\
            & := & \{(\y, \z) \in D_N \times D_N : \y \neq \z \ \hbox{and} \ \pi (\x) \sim \pi (\y) = \pi (\z) \}. \end{array} $$
 The dynamics is then described by the Markov generator
\begin{equation}
\label{eq:micro-model}
  \begin{array}{rcl}
   L_- f (\eta) & = & \displaystyle \sum_{\x} \ \ [f (\eta_{\x, 0}) - f (\eta)] \vspace*{2pt} \\ & + &
                      \displaystyle \sum_{\x} \ \bigg(\frac{a}{N (N - 1)} \sum_{(\y, \z) \in A (\x)} \eta (\y) \,\eta (\z) \bigg) \ [f (\eta_{\x, 1}) - f (\eta)] \vspace*{2pt} \\ & + &
                      \displaystyle \sum_{\x} \ \bigg(\frac{1}{2M} \ \frac{b}{N (N - 1)} \sum_{(\y, \z) \in B (\x)} \eta (\y) \,\eta (\z) \bigg) \ [f (\eta_{\x, 1}) - f (\eta)] \end{array}
\end{equation}
 where configuration~$\eta_{\x, i}$ is obtained from~$\eta$ by setting the state at~$\x$ equal to~$i$.
 We shall call this process the \emph{microscopic representation}.
 To study this model, it is convenient to also consider its \emph{mesoscopic representation} that keeps track of the metapopulation at the patch level rather than at the individual level.
 This new process is simply obtained by setting
 $$ \begin{array}{l} \xi_t (x) \ := \ \sum_{\x : \pi (\x) = x} \,\eta_t (\x) \quad \hbox{for all} \quad (x, t) \in \Z \times \R_+. \end{array} $$
 In words, the process counts the number of individuals at each patch.
 Note that, since the particular locations of the individuals within each patch is unimportant from a dynamical point of view, this new process is again a Markov process, and its Markov generator is given by
\begin{equation}
\label{eq:macro-model}
  \begin{array}{rcl}
   L_+ f (\xi) & = & \displaystyle \sum_x \ \xi (x) \,[f (\xi_{x-}) - f (\xi)] \vspace*{2pt} \\ & + &
                     \displaystyle \sum_x \ \frac{a}{N (N - 1)} \ \ \xi (x) \,(\xi (x) - 1)(N - \xi (x)) \ [f (\xi_{x+}) - f (\xi)] \vspace*{2pt} \\ & + &
                     \displaystyle \sum_x \ \sum_{y \sim x} \ \frac{1}{2M} \ \frac{b}{N (N - 1)} \ \ \xi (y) \,(\xi (y) - 1)(N - \xi (x)) \ [f (\xi_{x+}) - f (\xi)] \end{array}
\end{equation}
 where configuration $\xi_{x \pm}$ is obtained from $\xi$ by adding/removing one individual at $x$.
 The analog of this model derived from the basic contact process rather than the contact process with sexual reproduction has been studied in~\cite{bertacchi_lanchier_zucca_2011} where it is proved that
\begin{itemize}
 \item the process survives when~$a + b > 1$ and~$N$ is sufficiently large, \vspace*{2pt}
 \item the process dies out for all values of the parameter~$N$ when~$a + b \leq 1$.
\end{itemize}
 The analysis of the stochastic process~\eqref{eq:micro-model}--\eqref{eq:macro-model} is more challenging due to the complexity of the birth mechanism.
 Our approach is to use duality techniques to show that, at least in bounded space-time regions and when~$N$ is large, the stochastic process can be well approximated by the system of coupled
 differential equations for~$\u = (u_x)_{x \in \Z}$, called the \emph{mean-field equations}, given by
\begin{equation}
\label{eq:mean-field}
   u_x' \ = \ \bigg(a u_x^2 + \frac{b}{2M} \,\sum_{y \sim x} \,u_y^2 \bigg)(1 - u_x) - u_x \quad \hbox{for all} \quad x \in \Z.
\end{equation}
 The long-term behavior of the process can then be deduced from properties of the mean-field equations combined with block constructions. \vspace*{5pt}


\noindent{\bf The mean-field equations} --
 Starting from~$\u (0)$ constant, the profile~$\u (t)$ remains constant across space at all times and solves the single differential equation
\begin{equation}
\label{eq:smf}
  \u' \ = \ r \u^2 \,(1 - \u) - \u \quad \hbox{where} \quad r := a + b.
\end{equation}
 Some basic algebra shows that
\begin{itemize}
 \item when~$r < 4$, there is a unique equilibrium, namely~0, \vspace*{2pt}
 \item when~$r = 4$, we have the pair of equilibria: 0 and 1/2, \vspace*{2pt}
 \item when~$r > 4$, there are three equilibria: 0,
  $$u_- := 1/2 - w \quad \hbox{and} \quad u_+ := 1/2 + w \quad \hbox{where} \quad w = (1/4 - 1/r)^{1/2} $$
  with 0 and~$u_+$ stable and~$u_-$ unstable.
\end{itemize}
 This motivates the following definition, which is the key to understanding the system~\eqref{eq:mean-field} starting from more general profiles:
for~$r = a + b > 4$, we say that
\begin{itemize}
 \item {\bf expansion} occurs if there is~$u$ with~$u_- < u < u_+$ so that
   $$ \u (0) = u \,\mathbf{1} (x \leq 0) \quad \hbox{implies that} \quad u_1 (t_0) = u \ \ \hbox{for some} \ \ t_0 > 0, $$
 \item {\bf retreat} occurs if there are~$u_*$ and~$u^*$ with~$0 < u_* < u_- < u_+ <u^*$ so that
   $$ \u (0) = u^* \,\mathbf{1} (x < 0) + u_* \,\mathbf{1} (x \geq 0) \quad \hbox{implies that} \quad u_{-1} (t_0) = u_* \ \ \hbox{for some} \ \ t_0 > 0. $$
\end{itemize}
 Notice that retreat is also defined for $r = 4$ by letting $u_- = u_+ = 1/2$.
 Moreover, if~$r < 4$ we shall say that retreat occurs, since then the conclusion of the upcoming Theorem~\ref{th:ode-spread} holds for any~$u > 0$.
 Expansion implies that starting from a large enough occupied block of patches, the population spreads at a linear rate. Retreat implies the opposite:  starting from a large enough vacant block of patches, the population dies out within a region that grows linearly in time. 
 This is summarized by the following result, which can be deduced from more general results of \cite{weinberger_1982}. 
 We give a more direct proof in Section~\ref{sec:mean-field} using our duality theory.
\begin{theorem} --
\label{th:ode-spread}
 We have the following implications for expansion and retreat.
\begin{itemize}
 \item If expansion occurs, there are~$u, L, x_0, \delta, c > 0$ with~$u_- < u < u_+$ so that \vspace*{4pt}
 \item[] \hspace*{10pt} $u_x (0) \geq u \ \ \hbox{for all} \ \ |x| \leq L \quad \hbox{implies that} \quad u_x (t) \geq u + \delta \ \ \hbox{for all} \ \ |x| \leq ct - x_0$. \vspace*{10pt}
 \item If retreat occurs, there are $u, L, x_0, \delta, c > 0$ with $0 < u < u_-$ so that \vspace*{4pt}
 \item[] \hspace*{10pt} $u_x (0) \leq u \ \ \hbox{for all} \ \ |x| \leq L \quad \hbox{implies that} \quad u_x (t) \leq u - \delta \ \ \hbox{for all} \ \ |x| \leq ct - x_0$.
\end{itemize}
\end{theorem}
 It is not possible to obtain exact threshold values for expansion or retreat.
 However, we can obtain some decent estimates, which among other things implies that when~$M = 1$, retreat occurs for an open set of values of~$(a, b)$ satisfying~$a + b > 4$.
\begin{theorem} --
\label{th:expansion-retreat}
 Expansion and retreat are \emph{open} conditions in that the sets
 $$ \{(a, b) : \hbox{expansion occurs} \} \quad \hbox{and} \quad \{(a, b) : \hbox{retreat occurs} \} $$
 are open subsets of~$\R_+^* \times \R_+^*$.
 Moreover, if~$M = 1$, then
\begin{itemize}
 \item expansion occurs when~$a + b/2 > 4$ and~$b > 8/9$, \vspace*{4pt}
 \item retreat occurs when $a+b\leq 4$ and $b>0$.
\end{itemize}
\end{theorem}
 Section~\ref{sec:mean-field} also collects results for survival of the system of ordinary differential equations starting with a single fully occupied patch and all the other patches empty. \vspace*{5pt}


\noindent{\bf The stochastic process} --
 We now state our results for the stochastic process.
 Constructing the system graphically from a collection of independent Poisson processes and using standard coupling arguments, one easily proves monotonicity with respect to the birth rates~$a$ and~$b$, as well as attractiveness.
 Note however that basic coupling arguments do not imply monotonicity with respect to the patch capacity~$N$ or the dispersal range~$M$.
 Attractiveness implies in particular that the limiting distribution of the process starting from the all occupied configuration exists,
 so survival and extinction can be studied through this invariant measure looking at whether it is a nontrivial distribution or the point mass at the all empty configuration, that we denote from now on by~$\0$.
 Using duality techniques and block constructions as well as properties of the system of differential equations given by Theorem~\ref{th:expansion-retreat}, we can prove the following two results.
\begin{theorem} --
\label{th:expansion}
 Assume that expansion occurs in the mean-field equations.
 Then, for all~$N$ large, the stochastic process has a nontrivial stationary distribution, and there is~$L$ such that
 $$ \begin{array}{l} \lim_{N \to \infty} \,P \,(\xi_t = \0 \ \hbox{for some} \ t > 0 \ | \ \xi_0 (x) = N \ \textrm{for all} \ x \in [-L, L]) \ = \ 0. \end{array} $$
\end{theorem}
\begin{theorem} --
\label{th:retreat}
 Assume that retreat occurs in the mean-field equations.
 Then, for all~$N$ large, the stochastic process converges in distribution to the point mass at~$\0$.
\end{theorem}
 To answer an important ecological question, namely, whether an alien species established in one patch either successfully spreads in space or is doomed to extinction, we now study the probability
 of long-term survival for the process starting with a single fully occupied patch, that is the process starting from the initial configuration
 $$ \xi_0 (0) = N \quad \hbox{and} \quad \xi_0 (x) = 0 \quad \hbox{for all} \quad x \neq 0. $$
 From now on, we let~$P_o$ denote the law of the process starting from this configuration.
 Relying again on properties of the system~\eqref{eq:mean-field} together with duality techniques and block constructions, we obtain the following sufficient condition for successful invasion.  Note the conditions on the birth rates are slightly stronger than for expansion, since we want the population to spread starting from a \emph{single} occupied patch and not just from a sufficiently large finite block.
\begin{theorem} --
\label{th:spread}
 Assume that~$M = 1$. Then,
 $$ \begin{array}{l} \lim_{N \to \infty} \,P_o \,(\xi_t = \0 \ \hbox{for some} \ t > 0) \ = \ 0 \quad \hbox{whenever} \quad a + b/2 > 4 \ \hbox{and} \ b > 2. \end{array} $$
\end{theorem}
 Finally, we study he dependence of survival on the dispersal range~$M$.
 Theorem~\ref{th:retreat} shows that extinction occurs starting from any configuration when the birth parameters are small enough and the patch capacity is large.
 In contrast, our last theorem focuses on the process starting with a single fully occupied patch, in which case, regardless of the birth rates and the capacity of the patches, extinction occurs with probability close
 to one when the dispersal range is sufficiently large.
\begin{theorem} --
\label{th:range}
 For all $M$ large, we have
 $$ \begin{array}{rclcl}
      P_o \,(\xi_t \neq \0 \ \hbox{for all} \ t > 0) & \leq & M^{-1/3} \,(1/2 + b \,N (1 - a/4)^{-1})               & \hbox{when} & a < 4  \vspace*{3pt} \\
                                                                        & \leq & M^{-1/3} \,(1/2 + (b/2)(N + 2)^2)                     & \hbox{when} & a = 4  \vspace*{3pt} \\
                                                                        & \leq & M^{-1/3} \,(1/2 + b \,(a/4 - 1)^{-2} \,(a/4)^{N + 2}) & \hbox{when} & a > 4. \end{array} $$
\end{theorem}
 There are three different estimates because the survival probability is related to the time to extinction of a patch in isolation, which scales differently depending on whether the inner birth rate is subcritical,
 critical or supercritical.
 In either case, the theorem shows that the survival probability decreases to zero as the dispersal range increases to infinity, indicating that long range dispersal promotes extinction of metapopulations subject to a strong
 Allee effect caused by sexual reproduction.
 In particular, the effects of dispersal are somewhat opposite for the process with and without sexual reproduction since, as proved in~\cite{bertacchi_lanchier_zucca_2011}, in the presence of long range dispersal,
 the basic contact process approaches a branching process with critical values for survival significantly smaller than that of the process with nearest neighbor interactions. \\
\indent The paper is laid out as follows.
 In Section~\ref{sec:dual}, we introduce the dual to the stochastic model and show monotonicity of both.
 In Section~\ref{sec:dual-mean-field}, we introduce the dual to the mean-field equations and use it to prove some useful approximation results for solutions to the mean-field equations, that can be viewed as localized and quantitative statements
 of continuity with respect to initial and boundary data.
 In Section~\ref{sec:dual-coupling}, we show how to obtain the dual to the stochastic model by truncation of the dual to the mean-field equations, and conclude the two duals coincide so long as there is no collision of particles.
 In Section~\ref{sec:occupation}, we use the above agreement of duals to compare the occupation density at each patch in the stochastic model to the mean-field value, with an error term that is proportional to the collision probability.
 In Section~\ref{sec:mean-field}, we pause for a moment to establish the stated properties of the mean-field equations. This uses only the results of Section~\ref{sec:dual-mean-field}.
 In Section~\ref{sec:expansion}, we prove Theorem~\ref{th:expansion}, using a block construction and the occupation density estimates of Section~\ref{sec:occupation}.
 In Section~\ref{sec:retreat}, we prove Theorem~\ref{th:retreat}, using the same idea, but with larger blocks and with more care to ensure the establishment of a completely vacant zone that grows over time.
 In Section~\ref{sec:spread}, we prove Theorem~\ref{th:spread} in the same manner as Theorem~\ref{th:expansion}.
 Finally, in Section~\ref{sec:range-extinction}, we show that, with high probability when~$M$ is large, the initially occupied patch dies out before any two individuals born
 at that patch are sent to the same neighboring patch in order to obtain Theorem~\ref{th:range}.


\section{The dual process of the patch model}
\label{sec:dual}

\indent To prove Theorems~\ref{th:expansion}--\ref{th:spread} we need to approximate the stochastic process using the mean-field equations, and we achieve this by constructing dual processes for both systems, which we then relate to one another. 
 We refer to the stochastic process with patch size~$N$ as the~\emph{$N$-patch model}, and to its dual process as the~\emph{$N$-dual}.
 This will help to distinguish these dual processes from the dual of the mean-field equations introduced in the next section.
 The main objective of this section is to construct both the~$N$-patch model and its dual process from a graphical representation.
 Then, using this graphical representation, we prove that both processes are monotone with respect to the inner and outer birth rates~$a$ and~$b$.
 To define the graphical representation, for each microscopic spatial location~$\x$, we introduce the following random variables.
\begin{itemize}
 \item For each~$(\y, \z) \in A (\x)$, \vspace*{4pt}
 \item[] \ \ $\{a_n (\x, \y, \z) : n > 0 \} :=$ Poisson point process with rate~$a \,(N (N - 1))^{-1}$. \vspace*{4pt}
 \item For each~$(\y, \z) \in B (\x)$, \vspace*{4pt}
 \item[] \ \ $\{b_n (\x, \y, \z) : n > 0 \} :=$ Poisson point process with rate~$(b / 2M) \,(N (N - 1))^{-1}$. \vspace*{4pt}
 \item In addition, $\{d_n (\x) : n > 0 \} :=$ Poisson point process with rate~1.
\end{itemize}
 The microscopic process is constructed from these Poisson processes as follows:
\begin{itemize}
 \item {\bf Births}: at time~$t = a_n (\x, \y, \z)$ or~$t = b_n (\x, \y, \z)$, we set
  $$ \begin{array}{rclcl}
     \eta_t (\x) & = & 1              & \hbox{when} & \eta_{t-} (\y) \,\eta_{t-} (\z) = 1 \vspace*{3pt} \\
                 & = & \eta_{t-} (\x) & \hbox{when} & \eta_{t-} (\y) \,\eta_{t-} (\z) = 0. \end{array} $$
 \item {\bf Deaths}: at time $t = d_n (\x)$, we set $\eta_t (\x) = 0$.
\end{itemize}
 We also construct~$\hat \eta_s (\w, t)$, the dual process starting at space-time point~$(\w, t)$, to keep track of the state at this point based on the configuration at time~$t - s$.
 The dual process of the contact process with sexual reproduction consists of a collection of finite subsets of~$D_N$, the set of spatial locations.
 To describe its dynamics, it is convenient to introduce
 $$ \begin{array}{l} A_s^N := \{\x \in D_N : \x \in B \ \hbox{for some} \ B \in \hat \eta_s (\w, t) \} = \bigcup_{B \in \hat \eta_s (\w, t)} B \end{array} $$
 which we call the \emph{active set}.
\begin{enumerate}
\item The process starts from the singleton $\hat \eta_0 (\w, t) = \{\{\w \}\}$. \vspace*{2pt}
\item {\bf Births}: if site~$\x \in A_{s-}^N$ where either
 $$ s = t - a_n (\x, \y, \z) \quad \hbox{or} \quad s = t - b_n (\x, \y, \z) \quad \hbox{for some} \quad n > 0 $$
 then, for each set $B \in \hat \eta_{s-} (\w, t)$ that contains $\x$, we add the set which is obtained from~$B$ by removing $\x$ and adding its parents' sites $\y$ and $\z$, i.e.,
 $$ \hat \eta_s (\w, t) \ := \ \hat \eta_{s-} (\w, t) \ \cup \ \{(B - \{\x \}) \,\cup \,\{\y, \z \} : \x \in B \in \hat \eta_{s-} (\w, t) \}. $$
\item {\bf Deaths}: if site~$\x \in A_{s-}^N$ where~$s = t - d_n (\x)$ for some $n > 0$ then we remove from the dual process all the sets that contain $\x$, i.e.,
 $$ \hat \eta_s (\w, t) \ := \ \hat \eta_{s-} (\w, t) - \{B \in \hat \eta_{s-} (\w, t) : \x \in B \}. $$
\end{enumerate}
 See Figure~\ref{fig:dual-patch} for a picture.
 The dual process allows us to deduce the state of site $\w$ at time $t$ from the configuration at earlier times.
 More precisely, identifying the microscopic process with the set of occupied sites, the construction of the dual process implies the duality relationship
\begin{equation}
\label{eq:duality-1}
  \w \in \eta_t \quad \hbox{if and only if} \quad B \,\subset \,\eta_{t - s} \ \ \hbox{for some} \ \ B \in \hat \eta_s (\w, t).
\end{equation}
 Using a basic coupling argument, we now prove that the~$N$-patch model and the~$N$-dual are monotone with respect to the inner and outer birth rates.
\begin{lemma} --
\label{lem:monotone}
 The~processes~$\eta_t$ and~$\hat \eta_s (\w, t)$ are nondecreasing with respect to~$a$ and~$b$.
\end{lemma}
\begin{proof}
 Fix~$a_1 \leq a_2$ and~$b_1 \leq b_2$ and, for~$i = 1, 2$, let
 $$ \eta_t^i \ := \ \hbox{the~$N$-patch model with inner and outer birth rates~$a_i$ and~$b_i$}. $$
 Then, construct the first~$N$-patch model~$\eta_t^1$ from the graphical representation above with~$a = a_1$ and~$b = b_1$.
 Basic properties of Poisson processes imply that the process~$\eta_t^2$ can be constructed from the same graphical representation supplemented with additional independent Poisson processes with intensity~$a_2 - a_1$ for inner births
 and~$b_2 - b_1$ for outer births.
 This defines a coupling of the two~$N$-patch models for which one easily shows that
 $$ P \,(\eta_t^1 \subset \eta_t^2) = 1 \ \ \hbox{for all} \ \ t \geq 0 \quad \hbox{whenever} \quad \eta_0^1 \subset \eta_0^2. $$
 This shows that the~$N$-patch model is nondecreasing with respect to the inner and outer birth rates.
 The monotonicity of the~$N$-dual can be proved similarly.
\end{proof}
\begin{figure}[t]
\centering
\scalebox{0.40}{\input{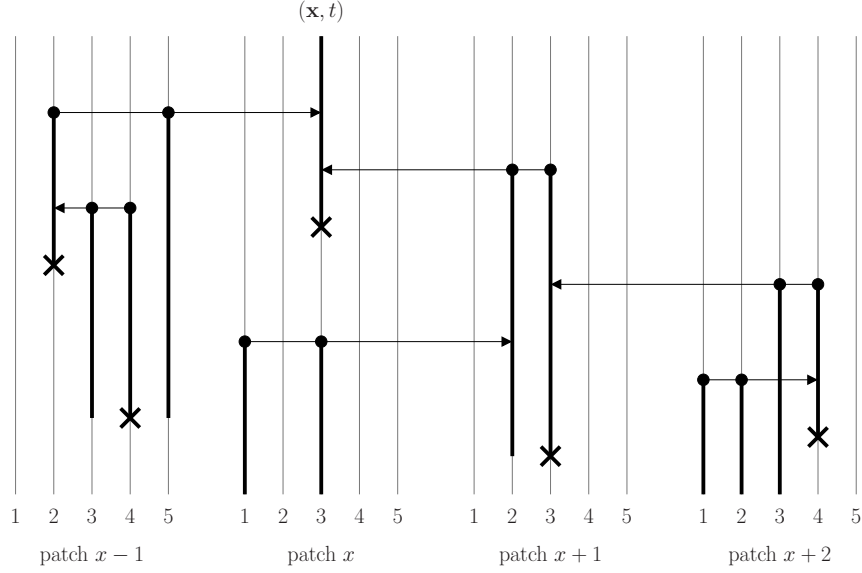}}
\caption{\upshape{Picture of the~$N$-dual for~$N = 5$ with time going up.
 Arrows represent birth events with the offspring at the head of the arrow and the parents' location being the two dots at the tail of the arrow, while crosses represent death events.
 The bold lines refer to the active set.
 In our realization, the dual at the bottom of the picture consists of only one set of cardinal five, so site~$\x$ is occupied at time~$t$ if and only if these five sites are occupied.}}
\label{fig:dual-patch}
\end{figure}

\section{The dual process of the mean-field equations}
\label{sec:dual-mean-field}

\indent We now discuss the dual of the mean-field equations, which we call the \emph{limiting dual} because it appears to be the the limit in distribution of the~$N$-dual, and denote by~$\zeta_t$.
 Its state space consists of the set of finite collections of finite sets of points in~$D := \Z \times (0, 1)$.
 As previously, the dynamics is described using the \emph{active set} given by
 $$ \begin{array}{l} A_t := \{(x, w) \in D : (x, w) \in B \ \hbox{for some} \ B \in \zeta_t \} = \bigcup_{B \in \zeta_t} B. \end{array} $$
 Then, the process $\zeta_t$ has the following transitions.
\begin{itemize}
\item For each~$(x, w) \in A_t$, at rate~$a$, add to~$\zeta_t$ all sets obtained from sets~$B \in \zeta_t$ such that~$(x, w) \in B$ by removing~$(x, w)$ and adding~$(x, w_1), (x, w_2)$ where~$w_1, w_2$ are independent uniform random variables on the interval~$(0, 1)$.\vspace*{4pt}
\item For each~$(x, w) \in A_t$ and each~$y \sim x$, at rate~$b/(2M)$, add to~$\zeta_t$ all sets obtained from sets~$B \in \zeta_t$ such that~$(x, w) \in B$ by removing~$(x, w)$ and adding $(y, w_1),(y, w_2)$ where~$w_1, w_2$ are independent uniform random random variables on~$(0, 1)$.\vspace*{4pt}
\item For each~$(x, w) \in A_t$, at rate 1, remove from~$\zeta_t$ all sets~$B$ containing~$(x, w)$.
\end{itemize}
 We now exhibit the connection between the limiting dual and the mean-field equations.
 Suppose~$\u$ is given.
 For~$B \in \zeta_t$, writing~$B  = \{(x_1, w_1), \ldots, (x_k, w_k) \}$, say that
 $$ B \ \hbox{is \emph{good} for} \ \u \quad \hbox{if and only if} \quad w_j \leq u_{x_j} \ \hbox{for all} \ j = 1, 2, \ldots, k $$
 and write~$\zeta_t \sim \u$ if there exists~$B \in \zeta_t$ that is good for~$\u$.
 Letting~$w \sim \uniform (0, 1)$, for~$t > 0$ and~$\u \in D$, define the function~$\phi : \mathbb{R}_+ \times \mathcal{K} \to \mathcal{K}$ where~$\mathcal{K} = [0, 1]^{\Z}$ by
\begin{equation}
\label{eq:phidef}
  (\phi_t (\u))_x \ = \ P \,(\zeta_t \sim \u \mid \zeta_0 = \{\{(x, w)\}\}).
\end{equation}
 We want to show that~$\phi_t$ gives us the solutions to the mean-field equations. 
 Our first task is to show that it has the semigroup property.
\begin{lemma} --
\label{lem:semigp}
 For~$\phi$ as defined in~\eqref{eq:phidef} and~$s, t > 0$, we have~$\phi_{t + s} = \phi_t \circ \phi_s$.
\end{lemma}
\begin{proof}
 For an illustration of what is introduced in the proof, we refer to Figure~\ref{fig:dual-limit}.
 First, we re-express~$\zeta_t$ as a labelled set of points in~$D$, namely as~$(I_t, \ell_t)$, where
\begin{itemize}
 \item for each~$t \geq 0$, the \emph{influence set}~$I_t$ is a finite set of points in~$D$ and \vspace*{4pt}
 \item for each~$t \geq 0$, the \emph{labelling}~$\ell_t : I_t \to \Z_+^3$ keeps track of the child-sibling-parent relation whenever the influence set~$I_t$ branches.
\end{itemize}
 The influence set starts at~$I_0 = \{(x,w)\}$ and has the following transitions.
\begin{itemize}
\item For each~$(x, w) \in I_t$, at rate~$a$, add to~$I_t$ the pair of points~$(x, w_1)$ and~$(x,w_2)$ where~$w_1,w_2$ are independent $\uniform (0, 1)$ random variables.\vspace*{4pt}
\item For each~$(x, w) \in I_t$ and each $y \sim x$, at rate $b/(2M)$, add to~$I_t$ the pair of points~$(y, w_1), (y, w_2)$ where~$w_1, w_2$ are independent $\uniform (0, 1)$ random variables. \vspace*{4pt}
\item For each~$(x, w) \in I_t$, at rate~1, remove $(x, w)$ from $I_t$.
\end{itemize}
 Since newly added points~$(x, w)$ have~$w \sim \uniform (0, 1)$, with probability one they do not collide with existing points, so $I_t$ is a branching random walk.
 Also, we have~$A_t \subset I_t$ under the obvious coupling, and strict inclusion is possible since the removal of a point from the influence set causes the removal of all sets~$B \in \zeta_t$ that contain this point.

\begin{figure}[t]
\centering
\scalebox{0.40}{\input{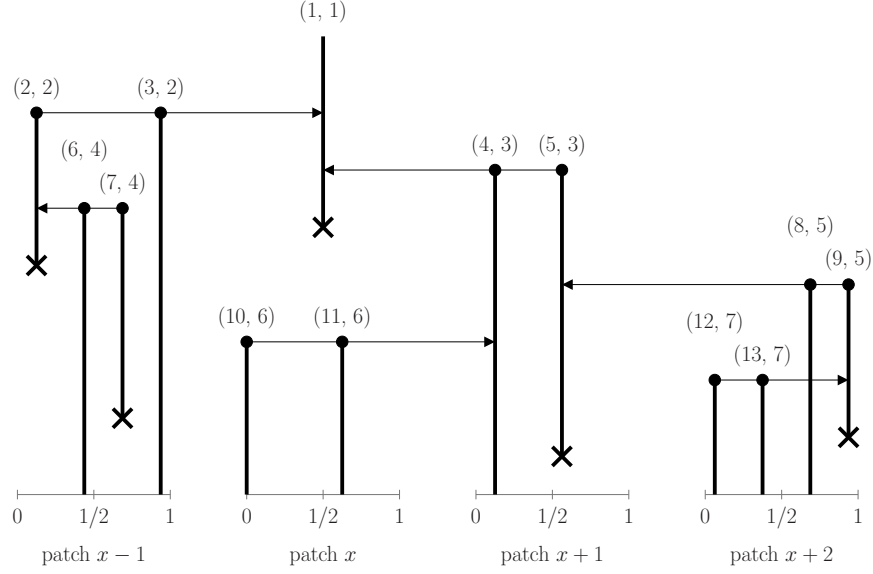}}
\caption{\upshape{Picture of the limiting dual with time going up.
 Arrows, dots and crosses have the same interpretation as in Figure~\ref{fig:dual-patch} while the bold lines now refer to the influence set.
 The pairs of numbers in the picture are respectively the personal label and generation number of each of the sites that are included in the influence set.}}
\label{fig:dual-limit}
\end{figure}

\indent The labelling is then defined as follows.
 Each point in~$I_t$ gets three labels: its own personal label, a generation number, and the label of its parent, with the parent label of the initial point being unimportant.
 If a point is removed from~$I_t$, its personal label and generation number are not reused, and the labelling of a given point does not change over time.
 Personal labels and generation numbers are in increasing order of appearance, so the single point in~$I_0$ gets the label~1 and generation number~1, and if personal labels~$1, \ldots, k$ and generation numbers~$1, \ldots, m$
 are in use and a new pair of points appears, then they are given personal labels~$k + 1$ and~$k + 2$, generation number~$m + 1$, and the label of their parent.

\indent Given $(I_s,\ell_s)_{0 \leq s \leq t}$, we now identify a set of labels active at time $t$ recursively as follows.  
 For each~$(x, w) \in I_t$, if~$w \leq u_x$, the label of~$(x,w)$ is active.
 If there are two active labels with the same generation number, their common parent label is active.
 Since the generation number of a parent is smaller than that of its child, and since, with probability one, after a finite time only a finite set of generation numbers have been assigned, if we repeat the last step, after a
 finite number of iterations all active labels have been found.
 The intuition behind this construction can be understood returning to the microscopic process as follows:
 whenever there is a birth event~$\{\y, \z \} \to \x$, the target location becomes occupied if both parent' locations are occupied.
 It follows by inspection that
 $$ \zeta_t \sim \u \quad \hbox{if and only if} \quad \hbox{label~1 is active at time~$t$}. $$

\indent We now observe the labelled influence set has the following composition property.  Given~$s, t > 0$ and $\{(I_r, \ell_r) : 0 \leq r \leq t \}$, each particle in $I_t$ evolves independently and in the same way as a
 single particle at time~0
 Therefore, we can construct~$\{(I_{t + r}, \ell_{t + r}) : 0 \leq r \leq s \}$ by first conditioning on~$I_t$ and then appending, to each~$(x, w) \in I_t$, an independent copy of the process started from~$(x, w)$.
 To maintain consistency with~$I_t$ it then suffices to re-label each copy while preserving the relation between the labels.
 The set of values~$\{w : (x, w) \in I_t\}$ are independent~$\uniform(0,1)$ random variables, so it follows that the labels of points~$(x,w) \in I_t$ are active at time $t+s$ independently with probability~$u_x(s) = (\phi_s (\u))_x$.
 Using the same fact, this can be equivalently stated by saying that the label of~$(x, w) \in I_t$ is active if~$w \leq u_x (s)$.
 Using the definition of~$\phi$, we find that the label of the single point~$(x, w) \in I_0$ is active with probability~$(\phi_t (\phi_s (\u)))_x$.
 Since it is also active with probability~$(\phi_{t + s} (\u))_x$, this completes the proof.
\end{proof} \\ \\
 We say that~$\phi_t$ is the \emph{flow} corresponding to~\eqref{eq:mean-field} if for each~$\u \in \mathcal{K}$, the function~$\u (t) = \phi_t (\u)$ is the unique solution to~\eqref{eq:mean-field} with~$\u (0) = \u$.
\begin{theorem} --
\label{th:flow}
 The function~$\phi_t$ is the flow corresponding to~\eqref{eq:mean-field}.
\end{theorem}
\begin{proof}
 We need to check that
\begin{equation}
\label{eq:flow}
  \partial_t \phi_t (\u) = F (\u) \quad \hbox{for all} \quad t \geq 0 \ \ \hbox{and} \ \ \u \in \mathcal{K}
\end{equation}
 where~$F (\u)$ is the right-hand side of~\eqref{eq:mean-field}.
 In light of Lemma~\ref{lem:semigp}, it suffices to show that equation~\eqref{eq:flow} holds when~$t = 0$ and~$\u \in \mathcal{K}$, since
 $$ \begin{array}{rcl}
     \partial_t \phi_t (\u) & = & \lim_{h \to 0^+} \,h^{-1} \,[\phi_{t + h} (\u) - \phi_t (\u)] \vspace*{4pt} \\
                            & = & \lim_{h \to 0^+} \,h^{-1} \,[\phi_h (\phi_t (\u)) - \phi_t (\u)] \ = \ \partial_s \phi_s (\phi_t (\u)) \big|_{s = 0} \end{array} $$
 and~$\phi_t (\u) \in \mathcal{K}$.
 In order to prove~\eqref{eq:flow}, we first observe that, given~$\zeta_0 = \{\{(x, w) \} \}$ and~$h > 0$ small, the list of all possible events for~$\zeta_h$ along with their probabilities are
\begin{itemize}
 \item $E_1 := \{\zeta_h = \{\{(x, w)\}, \{(x, w_1), (x, w_2) \} \} \}$ for some independent uniform~$w_1, w_2$, which occurs with probability~$ah + o (h)$, \vspace*{2pt}
 \item $E_{2, y} : = \{\zeta_h = \{\{(x, w) \}, \{(y, w_1), (y, w_2)\} \} \}$ for each~$y \sim x$ and for some independent uniform~$w_1, w_2$, which occurs with probability~$bh / (2M) + o (h)$, \vspace*{2pt}
 \item $E_3 := \{\zeta_h = \varnothing \}$, which occurs with probability~$h + o (h)$, and \vspace*{2pt}
 \item $E_4 := \{\zeta_h = \zeta_0 \}$, which occurs with probability $1 - (1 + a + b) h + o (h)$.
\end{itemize}
 In addition, since~$w, w_1, w_2$ are~$\uniform (0, 1)$,
\begin{itemize}
 \item[] $P \,(\zeta_h \sim \u \ | \,E_1) = P \,(w \leq u_x) + P \,(w > u_x, w_1 < u_x, w_2 < u_x) = u_x + (1 - u_x) \,u_x^2$, \vspace*{2pt}
 \item[] $P \,(\zeta_h \sim \u \ | \,E_{2, y}) = P \,(w \leq u_x) + P \,(w > u_x, w_1 < u_y, w_2 < u_y) = u_x + (1 - u_x) \,u_y^2$,  \vspace*{2pt}
 \item[] $P \,(\zeta_h \sim \u \ | \,E_3) = 0$, \vspace*{2pt}
 \item[] $P \,(\zeta_h \sim \u \ | \,E_4) = P \,(w \leq u_x) = u_x$, \vspace*{2pt}
 \item[] $P \,(\zeta_0 \sim \u) = P \,(w \leq u_x) = u_x$.
\end{itemize}
 Putting things together and noting the cancellation of $u_x \,[1 - (a + b) h]$,
 $$ \begin{array}{rcl}
      \partial_s \phi_s (\u) \big|_{s = 0} & = & \lim_{h \to 0^+} \,h^{-1} \,[P \,(\zeta_h \sim \u) - P \,(\zeta_0 \sim \u)] \vspace*{4pt} \\
                                           & = & a \,P \,(\zeta_h \sim \u \ | \,E_1) + b \,(2M)^{-1} \sum_{y \sim x} P \,(\zeta_h \sim \u \ | \,E_{2, y}) \vspace*{4pt} \\ && \hspace*{10pt}
                                                 - \ (1 + a + b) \,P \,(\zeta_h \sim \u \ | \,E_4) \vspace*{4pt} \\
                                           & = & a \,(1 - u_x) \,u_x^2 + b \,(2M)^{-1} \sum_{y \sim x} (1 - u_x) \,u_y^2 - u_x \ = \ F (\u) \end{array} $$
 from which the desired result follows.
\end{proof} \\ \\
 Theorem~\ref{th:flow} implies some useful properties for the mean-field equations.
 The following is a direct consequence of Theorem~\ref{th:flow} and the definition of~$\phi_t$ in terms of~$\zeta_t$.
\begin{corollary} --
\label{cor:mt}
 Let~$\u (t)$ and~$\v (t)$ be solutions to~\eqref{eq:mean-field} with respective parameter values~$(a_u, b_u)$ and~$(a_v, b_v)$.
 Then, if~$a_u \leq a_v$ and~$b_u \leq b_v$,
 $$ \u (0) \leq \v(0) \quad \hbox{implies that} \quad \u (t) \leq \v(t) \ \ \hbox{for all} \ \ t \geq 0. $$
\end{corollary}
 In the next lemma, we control the size of the influence set.
\begin{lemma} --
\label{lem:size}
 Let~$J_s := \{x \in \Z : (x, w) \in I_s \}$ and~$\gamma > 0$. Then,
 $$ P \,(J_s \subset (- \infty, ct] \ | \ I_0 = \{(0, w_0) \}) \ \geq \ 1 - e^{- \gamma t} \quad \hbox{for all \ $s \leq t$ \ and some \ $c > 0$}. $$
\end{lemma}
\begin{proof}
 For~$\theta \in \R$ and~$t \geq 0$, define
 $$ \begin{array}{l} m_t (\theta) :=  E \,(\sum_{(y, v) \in I_t} e^{\theta y} \,| \,I_0 = \{(0, w_0) \}) \end{array} $$
 which encodes information about the distribution on~$\Z$ of the points in~$I_t$.
 By first conditioning on the value of~$I_t$ and noting that individual points evolve independently, we find
 $$ \partial_t m_t (\theta) \ = \ m_t (\theta) \,\partial_s m_s (\theta) \big|_{s = 0} $$
 and we readily compute
 $$ \begin{array}{l} \ell (\theta) := \partial_s m_s (\theta) \big|_{s = 0} \ = \ a + (2M)^{-1} \,\sum_{y \sim 0} \,b e^{\theta y} - 1 \end{array} $$
 which, since~$m_0 (\theta) = 1$, implies~$m_t (\theta) = e^{\ell (\theta) t}$.
 To control the spread of~$I_t$, let
 $$ \begin{array}{l} S_s^t (c) := |\{(y, v) \in I_s : y > ct \}| \ = \ \sum_{(y, v) \in I_s} \mathbf{1} (y > ct) \quad \hbox{for all} \quad c \in \R. \end{array} $$
 Since~$S_s^t (c)$ is integer valued and~$\mathbf{1} (y > ct) \leq e^{\theta y - \theta ct}$ for each~$\theta \in \R_+$,
\begin{equation}
\label{eq:size-1}
  \begin{array}{l}
    P \,(S_s^t (c) > 0 \ | \ I_0 = \{(0, w_0) \}) \vspace*{4pt} \\ \hspace*{25pt} = \
      \sum_{n \in \N^*} \,P \,(S_s^t (c) = n \ | \ I_0 = \{(0, w_0) \}) \ \leq \ E \,(S_s^t (c) \ | \ I_0 = \{(0, w_0) \}) \vspace*{4pt} \\ \hspace*{25pt} \leq \
    E \,(\sum_{(y, v) \in I_s} e^{\theta y - \theta ct} \ | \ I_0 = \{(0, w_0) \}) \ = \ e^{-\theta ct} \,m_s(\theta). \end{array}
\end{equation}
 Now, we observe that, when~$b > 0$, there exists~$\theta$ large such that~$\ell (\theta) \geq 0$, in which case~$m_s (\theta)$ is non-decreasing in~$s$.
 In particular, using~\eqref{eq:size-1}, we get, for~$c > \ell (\theta)$,
\begin{equation}
\label{eq:size-2}
  \begin{array}{rcl}
    P \,(S_s^t (c) > 0 \ | \ I_0 = \{(0, w_0) \}) & \leq & e^{- \theta ct} \,m_s (\theta) \vspace*{4pt} \\
                                                  & \leq & e^{- \theta ct} \,m_t (\theta) \ \leq \ e^{- \theta ct + \ell (\theta) t} \ = \ e^{-\gamma(\theta) t} \end{array}
\end{equation}
 for all~$s \leq t$, where~$\gamma (\theta) = \theta c - \ell (\theta) > 0$.
 On the other hand, when~$b = 0$, the projection of the influence set on~$\Z$ reduces to a singleton therefore~$S_s^t (c) = 0$ and~\eqref{eq:size-2} is trivial.
 Finally, we use that~$\gamma (\theta) \to \infty$ as~$c \to \infty$ together with~\eqref{eq:size-2} to conclude that, for all~$\gamma > 0$,
 $$ \begin{array}{l}
      P \,(J_s \subset (- \infty, ct] \ | \ I_0 = \{(0, w_0) \}) \vspace*{4pt} \\ \hspace*{25pt} = \ P \,(S_s^t (c) = 0 \ | \ I_0 = \{(0, w_0) \}) \ \geq \ 1 - e^{- \gamma t} \quad \hbox{for all} \quad s \leq t \end{array} $$
 for some~$c > 0$.
 This completes the proof.
\end{proof} \\ \\
 Lemma~\ref{lem:size} has the following pleasant consequence for~\eqref{eq:mean-field}.
\begin{lemma} --
\label{lem:mfappr}
 There exists $c, \gamma > 0$ so that if $\u (t), \v(t)$ are solutions of \eqref{eq:mean-field},
 $$ \begin{array}{l}
      u_x (0) = v_x (0) \ \ \hbox{for all} \ \ x \in [y - ct, y + ct] \vspace*{4pt} \\ \hspace*{100pt} \hbox{implies that} \quad |u_y (s) - v_y (s)| \leq 2 e^{-\gamma t} \ \ \hbox{for all} \ \ s \leq t. \end{array} $$
 In particular, for each $\ep, t > 0$, there is $ L > 0$ such that
 $$  \begin{array}{l}
       u_x (0) = v_x (0) \ \ \hbox{for all} \ \ x \in [y - L, y + L] \vspace*{4pt} \\ \hspace*{100pt} \hbox{implies that} \quad |u_y (s) - v_y (s)| < \ep \ \ \hbox{for all} \ \ s \leq t. \end{array} $$
\end{lemma}
\begin{proof}
 Let~$A = \{x \in \Z : u_x = v_x \}$ and~$s \leq t$.
 Then, by definition of~$\phi_s$,
\begin{equation}
\label{eq:mfappr-1}
  P \,(\zeta_s \sim \u \mid J_s \subset A) \ = \ P \,(\zeta_s \sim \v \mid J_s \subset A).
\end{equation}
 Using Lemma~\ref{lem:size} and~\eqref{eq:mfappr-1}, we deduce that there exist~$c, \gamma > 0$ such that
 $$ \begin{array}{rcl}
      |u_y (s) - v_y (s)| & = & |(\phi_s (\u))_y - (\phi_s (\v))_y| \ \leq \ P \,(J_s \cap A^c \neq \varnothing \mid I_0 = \{(y, w) \}) \vspace*{4pt} \\
                          & \leq & P \,(J_s \not \subset (- \infty, y + ct] \neq \varnothing \mid I_0 = \{(y, w) \}) \vspace*{4pt} \\ && \hspace*{50pt} + \
                                   P \,(J_s \not \subset [y - ct, \infty) \neq \varnothing \mid I_0 = \{(y, w) \}) \vspace*{4pt} \\
                          & \leq & 2 \,P \,(J_s \not \subset (- \infty, ct] \ | \ I_0 = \{(0, w_0) \}) \ \leq \ 2 \,e^{- \gamma t} \end{array} $$
 which proves the first part.
 The second statement is an easy consequence of the first.
\end{proof} \\ \\
 In order to help with block constructions later on, we conclude this section with two more estimates about truncated versions of our models.
\begin{definition}[restriction] --
\label{def:restriction}
  For integer~$K > 0$ and fixed but arbitrary boundary values~$u_x (t)$, $K < |x| \leq K + M$, $t \geq 0$ that are continuous functions of~$t$, we define the restrictions of the limiting dual~$^K\zeta_t$ and mean-field model~$^Ku_x (t)$ as follows.
\begin{itemize}
 \item Given~$\zeta_0$ such that~$J_0 \subset [-K, K]$, the restriction~$^K\zeta_t$ is obtained from~$\zeta_t$ by freezing the evolution of any point~$(y, w) \in I_t$ that lands in~$|y| > K$ and declaring it good if~$w \leq u_y (t)$. \vspace*{4pt}
 \item Given~$\u (0)$, the restriction~$^Ku_x (t)$ is defined to be the solution to~\eqref{eq:mean-field} where only the values~$u_x (t)$ with~$|x| \leq k$ are determined by \eqref{eq:mean-field}, and using the given boundary
  values~$u_y (t)$ if~$y \sim x$ and~$|y| > K$.
\end{itemize}
\end{definition}
 The two most obvious choices of boundary values are $u_y(t) \equiv 0$ and $u_y(t) \equiv 1$, that we call the \emph{lower} and \emph{upper} boundary values.
 Defining~$^K\phi_t(\u)$ as before but in terms of~$^K\zeta_t$, it follows that
 $$ ^K\phi_t(\u (0)) = \ ^Ku_x(t) \quad \hbox{for all} \quad x \in [-K, K] \ \ \hbox{and} \ \ t > 0. $$
 One verifies also that with lower boundary values, ~$^K\zeta_0 \subset \zeta_0$ implies~$^K\zeta_t \subset \zeta_t$ for~$t > 0$, and the same holds for upper boundary values but with the inclusion reversed.
 Since the event~$\zeta \sim \u$ is increasing with~$\zeta$ in the sense that if~$\zeta \sim \u$ and~$\zeta \subset \zeta'$ then~$\zeta'\sim \u$, for lower boundary values we obtain that
 $$ ^K\u(0) \leq \u (0) \ \ \hbox{on} \ \ [-K, K] \quad \hbox{implies} \quad ^K\u (t) \leq \u (t) \ \ \hbox{on} \ \ [-K, K] \ \ \hbox{for all} \ \ t > 0 $$
 and the same holds for upper boundary values but with the inequality reversed.
 The following result fills in the other side of the inequality in both situations.
\begin{lemma} --
\label{lem:block-estimate-fixed}
 Setting $^Ku_x(0) = u_x(0)$ for $|x| \leq K$, for each~$x \in [-K, K]$ and~$T > 0$, for lower boundary values we have
\begin{equation}
\label{eq:block-estimate-fixed-1}
  0 \leq u_x (t) - {}^Ku_x(t) \leq P \,(J_s \not \subset [-K, K] \textrm{ for some } s \in [0, t]).
\end{equation}
and the same is true for upper boundary values if we exchange $u_x(t)$ and $^Ku_x(t)$.
 In particular, for fixed $L, T$,
\begin{equation}
\label{eq:block-estimate-fixed-2}
  ^Ku_x(t) \rightarrow u_x (t) \textrm{ as } K \to \infty \quad \hbox{uniformly for} \ \ x \in [-L, L] \ \ \hbox{and} \ \ t \in [0, T].
\end{equation}
\end{lemma}
\begin{proof}
 We consider lower boundary values but the proof is analogous in the other case.
 That~$u_x(t)- {}^Ku_x(t) \geq 0$ follows from the discussion just above about monotonicity property of the limiting dual.
 The other inequality in~\eqref{eq:block-estimate-fixed-1} follows from the inclusions
 $$ \begin{array}{l}
    \{J_s \subset [-K, K] \ \hbox{for all} \ s \in [0, t] \} \,\cap \,\{\zeta_t \sim \u \} \ \subset \
    \{^K\zeta_t = \zeta_t \} \,\cap \,\{\zeta_t \sim \u \} \ \subset \ \{^K\zeta_t \sim \u \}. \end{array} $$
 For the second statement~\eqref{eq:block-estimate-fixed-2}, notice that from~$|I_0| = 1$ the cardinality of the set
 $$ \{(x, w) \in D : (x, w) \in I_s \textrm{ for some } s \in [0, t] \}$$
 is dominated by a branching process~$Z_t$ with~$Z_0 = 1$ in which each particle gives birth to a pair of particles at rate~$a + b$.
 Now, let~$\zeta_0 = \{(x, w) \}$ where~$x \in [-L, L]$, and let~$t \in [0, T]$.
 Then, since the diameter of the set~$J_s$ can only increase by~$M$ at each birth event,
 $$ \begin{array}{rcl}
      P \,(J_s \subset [-K, K] \ \hbox{for all} \ s \in [0, t]) & \geq &
      P \,(Z_t \leq (K - L)/M) \ \geq \ P \,(Z_T \leq (K - L)/M) \vspace*{4pt} \\ & = &
      1 - P \,(Z_T > (K - L)/M) \ \geq \ 1 - M \,(K - L)^{-1} \,E \,(Z_T) \vspace*{4pt} \\ & = &
      1 - M \,(K - L)^{-1} \,e^{(a + b) T} \ \to \ 0 \ \ \hbox{as} \ \ K \to \infty. \end{array} $$
 This shows statement~\eqref{eq:block-estimate-fixed-2} and completes the proof of the lemma.
\end{proof} \\ \\ 
 Lemma~\ref{lem:block-estimate-fixed} gives only rough bounds on the rate of convergence of~$^Ku_x(t)$ as~$K \to \infty$.
 With some extra work we get a better bound, that will be necessary in the proof of Theorem~\ref{th:retreat}.
 Note the~$n$ below has nothing to do yet with the patch size~$N$, though it will later on.
\begin{lemma} --
\label{lem:block-estimate-growing}
 For variable~$n > 0$ and fixed~$\alpha_1,\alpha_2$, let~$T := \alpha_1 \log n$ and~$L := \alpha_2 \log n$.
 Then, for each~$\kappa > 0$, there is~$\alpha_3$ and~$n_0$ so that, for~$K = \alpha_3\log n$,
 $$ |^Ku_x (t) - u_x (t)| \leq n^{-\kappa} \quad \hbox{for all} \quad n \geq n_0, \ |x| \leq L \ \ \hbox{and} \ \ t \in [0, T]. $$
\end{lemma}
\begin{proof}
 This follows from the first part of Lemma~\ref{lem:block-estimate-fixed}.
 In order to find a good upper bound for the right-hand side of~\eqref{eq:block-estimate-fixed-1}, we proceed in two steps by first controlling the growth of the set~$J_s$ in a short time interval and then at some regularly distributed times. \vspace*{5pt} \\
{\bf Step~1} -- Let~$Z_t$ be as in the proof of Lemma \ref{lem:block-estimate-fixed} and let~$\lambda := a + b$. Then,
 $$ P \,(Z_T > n^{\alpha_1 \lambda + \beta}) \ = \ P \,(Z_T > n^{\beta} \,e^{\lambda T}) \ = \ P \,(Z_T > n^{\beta} \,E \,(Z_T)) \ \leq \ n^{- \beta} $$
 for all~$\beta > 0$.
 In particular, letting~$\ep := \alpha_1 \lambda + \beta$, it follows that
\begin{equation}
\label{eq:block-estimate-growing-1}
  P \,(B) \ \geq \ 1- n^{-\beta} \quad \hbox{where} \quad B := \{|I_t| \leq n^{\ep} \textrm{ for all } t \in [0, T] \}.
\end{equation}
 In addition, letting~$X_t$ count the number of transitions occurring in the influence set up to time~$t$, we have the following stochastic domination:
 $$ P \,(\{X_{t + h} - X_t \geq k \} \cap B) \ \leq \ P \,(\poisson (n^{\beta} s) \geq k) \quad \hbox{for all} \ \ t < t + s < T. $$
 Letting~$h = n^{-\delta}$ with~$\delta > \ep > \beta$, we deduce that
\begin{equation}
\label{eq:block-estimate-growing-2}
  \begin{array}{rcl}
    P \,(\{X_{t + h} - X_t \geq k \} \cap B) & \leq & P \,(\poisson (n^{\beta} h) \geq k) \ \leq \ P \,(\poisson (n^{\ep - \delta}) \geq k) \vspace*{4pt} \\
                                             & \leq & \sum_{j \geq k} \,e^{- n^{\ep - \delta}} \,(n^{\ep - \delta})^j / j! \ \leq \ (n^{\ep - \delta})^k \,(1 - n^{\ep - \delta})^{-1} \end{array}
\end{equation}
 for any~$0 \leq t \leq T - h$. \vspace*{5pt} \\
{\bf Step~2} -- We now look at the process~$J_s$ at the times in
 $$ S \ := \ \{0 < t \leq T : t = jh \ \hbox{for some} \ j \in \N \}. $$
 Let~$I_0 = \{(x,w)\}$ for some~$|x| \leq L$ and~$J_{s, x}$ be the corresponding values~$J_s$.
 Then, it follows from Lemma~\ref{lem:size} that, for all~$\gamma > 0$, there is~$c > 0$ such that, for~$s \leq t$,
 $$ \begin{array}{rcl}
      P \,(J_{s, x} \not \subset [-L - ct, L + ct]) & \leq & P \,(J_{s, x} \not \subset [x - ct, x + ct]) \vspace*{4pt} \\
                                                    & \leq & 2 \,P \,(J_{s, x} \not \subset (- \infty, x + ct]) \ \leq \ 2 \,e^{- \gamma t}. \end{array} $$
 Replacing~$t$ with~$T$ and letting~$K_0 := L + cT$, for every~$m > 0$,
 $$ P \,(J_{t,x} \not \subset [- K_0, K_0]) \ \leq \ n^{-m} \quad \hbox{for all} \quad t \leq T $$
 for some~$c > 0$.
 In particular, for~$m > \delta$,
\begin{equation}
\label{eq:block-estimate-growing-3}
  P \,(J_{t, x} \not \subset [- K_0, K_0] \ \hbox{for some} \ t \in S) \ \leq \ |S| \,n^{-m} \ \leq \ n^{-(m - \delta)} \,\alpha_1 \log n
\end{equation}
 for some~$c > 0$. \vspace*{5pt} \\
{\bf Conclusion} -- Let~$k$ so that~$k (\delta - \ep) > \delta$ and~$K := K_0 + kM$ where~$M$ is the dispersal range.
 Using~\eqref{eq:block-estimate-growing-2} at all times~$t \in S$ together with the estimates~\eqref{eq:block-estimate-growing-1} and~\eqref{eq:block-estimate-growing-3}, we find
 $$ \begin{array}{l}
      P \,(J_{t, x} \subset [-K, K] \ \hbox{for all} \ t \in [0, T]) \vspace*{4pt} \\ \hspace*{40pt} \geq \
      1 - P \,(B^c) - P \,(\{J_{t, x} \not \subset [-K, K] \ \hbox{for some} \ t \in [0, T] \} \cap B) \vspace*{4pt} \\ \hspace*{40pt} \geq \
      P \,(B) - P \,(J_{t, x} \not \subset [-K_0, K_0] \ \hbox{for some} \ t \in S) - |S| \,P \,(\{X_{t + h} - X_t \geq k \} \cap B) \vspace*{4pt} \\ \hspace*{40pt} \geq \
      1 - n^{- \beta} - n^{-(m - \delta)} \,\alpha_1 \log n - |S| \,(n^{\ep - \delta})^k \,(1 - n^{\ep - \delta})^{-1} \vspace*{4pt} \\ \hspace*{40pt} \geq \
      1 - n^{- \beta} - \alpha_1 \log n \,(n^{-(m - \delta)} + n^{\delta} \,(n^{\ep - \delta})^k \,(1 - n^{\ep - \delta})^{-1}). \end{array} $$
 Using Lemma \ref{lem:block-estimate-fixed}, we deduce that, for~$x \in [-L, L]$ and~$t \in [0, T]$,
 $$ |^Ku_x (t) - u_x (t)| \ \leq \ n^{- \beta} + \alpha_1 \log n \,(n^{-(m - \delta)} + n^{- (k (\delta - \ep) - \delta)} \,(1 - n^{\ep - \delta})^{-1}). $$
 To conclude, it suffices to take~$\delta$, then~$m$, then~$k$ such that
 $$ \beta > \kappa \quad \hbox{and} \quad m - \delta > \kappa \quad \hbox{and} \quad k (\delta - \ep) - \delta > \kappa. $$
 This completes the proof.
\end{proof}


\section{Coupling of the dual processes, collision estimates}
\label{sec:dual-coupling}

\indent We now show how to couple the~$N$-duals to the limiting dual.
 We begin by constructing the limiting dual a bit more explicitly.
 For integer $k>0$, define the following random variables.
\begin{itemize}
 \item Let~$\{a_n (k), u_1 (n, k), u_2 (n, k) : n > 0 \} :=$ Poisson point processes with rate~$a$, and a pair of independent~$\uniform (0, 1)$ random variables attached to each Poisson event. \vspace*{4pt}
 \item For $0 < |m| \leq M$, let~$\{b_n (k, m), u_1 (n, k, m), u_2 (n, k, m) : n > 0 \} :=$ Poisson point process with rate~$b / 2M$, and a pair of independent~$\uniform (0, 1)$ random variables attached to each Poisson event. \vspace*{4pt}
 \item Let~$\{d_n (k) : n > 0 \} :=$ Poisson point process with rate~1.
\end{itemize}
 In case~$(x, w) \in I_{t^-}$ has personal label~$k$, then
\begin{itemize}
 \item at time~$t = a_n (k)$, set~$I_t = I_{t^-} \cup \{(x, u_1 (n, k)), (x, u_2 (n, k)) \}$, \vspace*{4pt}
 \item at time~$t = b_n (k, m)$, set~$I_t = I_{t^-} \cup \{(x + m, u_1 (n, k)), (x + m, u_2 (n, k)) \}$, and \vspace*{4pt}
 \item at time~$t = d_n (k)$, set~$I_t = I_{t^-} \setminus (x, w)$,
\end{itemize}
 with the labels of newly added points assigned as described earlier. \\
\indent We now construct a copy of the~$N$-dual using the same random variables.
 To distinguish it from the limiting dual, we use the notation~$\zeta_t^N$, $I_t^N$, etc.
 We first divide the interval~$[0, 1)$ into the~$N$ subintervals~$[(j - 1) / N, j / N)$.
 Then, to each point in~$I_t^N$, we add a fourth label, which we call the \emph{location label} and is assigned as follows. \vspace*{5pt} \\
 In case~$(y,v)$ is a new point in~$I_t^N$ and location labels~$1, 2, \ldots, k$ are in use,
\begin{itemize}
 \item if there is~$(x, w) \in I_t^N$ with~$x = y$ such that both second coordinates~$v$ and~$w$ lie in the same subinterval of~$[0, 1)$ then~$(y, v)$ is assigned the same location label as~$(x, w)$, \vspace*{4pt}
 \item otherwise~$(y, v)$ is assigned location label~$k + 1$.
\end{itemize}
 To construct the~$N$-dual from the previous random variables, we then make a slight adjustment to the above transitions, namely, if~$(x, w) \in I_{t^-}^N$ has location label~$k$, then
\begin{itemize}
\item at time $t = a_n (k)$, set~$I_t^N = I_{t^-}^N\cup \{(x, u_1 (n, k)),(x, u_2 (n, k)) \}$, \vspace*{4pt}
\item at time $t = b_n (k, m)$, set~$I_t^N = I_{t^-}^N\cup \{(x + m, u_1 (n, k)), (x + m, u_2 (n, k)) \}$, and \vspace*{4pt}
\item at time $t = d_n (k)$, all points in~$I_{t^-}^N$ with location label~$k$ are removed from~$I_t^N$.
\end{itemize}
 Also, when a new pair of points is added to~$I_t^N$, with~$u_1 (n, k)$ and~$u_2 (n, k)$ being the uniform random variables attached to the Poisson event occurring at time~$t$, their parent labels (now a subset of the integers~$\Z$)
 consist of the personal labels of all points with location label~$k$.
\begin{definition}[collision] --
 We say that a collision occurs if a newly created point in~$I_t^N$ is assigned the same location label as an existing point.
\end{definition}
 Note that, in the absence of any collisions, each point's location label is identical to its personal label, and so the~$N$-dual and the limiting dual coincide.
 In particular, it is useful to estimate
 $$ \tau^N \ := \ \hbox{time of the first collision starting from the pair~$\zeta_0^N = \{(x, w_1), (x, w_2) \}$.} $$
 Since the time of the first collision starting from a single point is stochastically larger, the following estimate holds for such a time as well.
\begin{lemma} --
\label{lem:collision-estimate}
 We have~$P \,(\tau^N \leq t) \leq (2 \,e^{2(a + b) t} + 1) \,N^{-1/3} \to 0$ as~$N \to \infty$.
\end{lemma}
\begin{proof}
 The idea is to first control~$n_t^N$, which is the number of points added to~$I_s^N$ in the time interval~$s \in [0,t]$ including the two initial points at time~0, and then show that, as long as this number
 of points is not too large, the probability of a collision is small.
 Noting that~$n_t^N$ is dominated stochastically by a branching process~$Z_t$ starting with two particles and in which each particle gives birth to two new particles independently at rate~$\lambda := a + b$, for each~$\gamma$ we have
\begin{equation}
\label{eq:collision-estimate-1}
  P \,(n_t^N > e^{\gamma t}) \ \leq \ P \,(Z_t > e^{\gamma t}) \ \leq \ e^{- \gamma t} \,E \,(Z_t) \ \leq \ 2 \,e^{(2 \lambda - \gamma) t}.
\end{equation}
 In other respects, if a new pair of points is added at time~$t$,
 $$ P \,(\hbox{collision at time~$t$} \ | \ n_{t-}^N \leq n) \ \leq \ 1 - (1 - n/N)(1 - (n + 1)/N) $$
 from which it follows that
\begin{equation}
\label{eq:collision-estimate-2}
  \begin{array}{l} P \,(\tau^N \leq t \ | \ n_t^N \leq n) \ \leq \ (1/N) \,\sum_{i \leq n} \,i \ \leq \ n^2 / N. \end{array}
\end{equation}
 Combining~\eqref{eq:collision-estimate-1}--\eqref{eq:collision-estimate-2} with~$n = e^{\gamma T}$, we get
 $$ \begin{array}{rcl}
      P \,(\tau^N \leq t) & \leq & P \,(n_t^N > e^{\gamma t}) + P \,(\tau^N \leq t \ | \ n_t^N \leq e^{\gamma t}) \vspace*{4pt} \\
                          & \leq & 2 \,e^{2 \lambda t} \,e^{- \gamma t} + (1/N) \,e^{2\gamma t}. \end{array} $$
 The lemma follows by taking~$\gamma = (1/3t) \ln N$ in the previous inequality.
\end{proof}
%


\section{Occupation density}
\label{sec:occupation}

\indent For the $N$-dual, we define the function~$\Phi_t^N$ in the same way as~$\phi_t$ has been defined in~\eqref{eq:phidef} but using the process~$\zeta_t^N$ instead of~$\zeta_t$.
 Then, it follows from the above that
\begin{equation}
\label{eq:phitruncomp-1}
  |\phi_t (\u) - \Phi_t^N (\u)| \ \leq \ P \,(\tau^N \leq t) \quad \hbox{for any} \quad \u \in K
\end{equation}
 since the limiting dual and the~$N$-dual coincide as long as there is no collision.
 In case~$\zeta_0$ has more than one point, we write that~$\zeta_t \sim \u$ when the label of \emph{every} point in~$\zeta_0$ is active at time~$t$ with respect to~$\u$.
 Now, define the functions
 $$ (\phi_t^2 (\u))_x \ := \ P \,(\zeta_t \sim \u \mid \zeta_0 = \{\{(x, w_1), (x, w_2) \} \}) $$
 where~$w_1$ and~$w_2$ are independent~$\uniform (0, 1)$, and similarly~$\Phi_t^{2, N}$ using~$\zeta_t^N$.
 Since the sets evolving from the two points~$(x, w_1)$ and~$(x, w_2)$ do so independently in the limiting dual, it follows that~$\phi_t^2 (\u) = (\phi_t (\u))^2$ as the notation suggests.
 Moreover,
\begin{equation}
\label{eq:phitruncomp-2}
  |\phi_t^2 (\u) - \Phi_t^{2, N} (\u)| \ \leq \ P \,(\tau^N \leq t)
\end{equation}
 which motivates our definition of $\tau^N$ earlier.
 Our next result is concerned with the occupation density of the~$N$-patch model, which we define as~$u_x^N (t) := (1/N) \,\xi_t (x)$.
\begin{theorem} --
\label{th:occupation-density}
 For all~$\ep > 0$ and~$x \in \Z$, we have
 $$ P \,(|u_x^N (t) - (\Phi_t^N (\u^N (0)))_x| > \ep) \ \leq \ 2 \,\ep^{-2} \ P \,(\tau^N \leq t). $$
\end{theorem}
\begin{proof}
 We first prove the result subject to the symmetry assumption
\begin{equation}
\label{eq:occupation-density-1}
  P \,(\eta_0 ((x, j_i)) = 1 \ \hbox{for} \ i = 1, \ldots, k) \ = \ P \,(\eta_0 ((x, \sigma(j_i))) = 1 \ \hbox{for} \ i = 1, \ldots, k)
\end{equation}
 for every subset~$\{j_1, j_2, \ldots, j_k\} \subset \{1,...,N \}$ and any permutation~$\sigma$ of this set.
 Note that if the previous equation holds for the initial configuration~$\eta_0$, then symmetry of the evolution rules implies that the same is true of~$\eta_t$ for all~$t > 0$.
 Using the duality relationship for the patch model together with the previous symmetry assumption~\eqref{eq:occupation-density-1} with~$k = 1$, we get
\begin{equation}
\label{eq:occupation-density-2}
  \begin{array}{l} E \,(\xi_t (x)) \ = \ \sum_{\x : \pi (\x) = x} \,P \,(\eta_t (\x) = 1) \ = \ N \,\Phi_t^N (\u^N (0)). \end{array}
\end{equation}
 Using the duality relation and~\eqref{eq:occupation-density-1} with~$k = 2$, we get
\begin{equation}
\label{eq:occupation-density-3}
  \begin{array}{l} E \,(\xi_t (x)^2) \ = \ \sum_{\x, \y : \pi (\x) = \pi(\y) = x} \,P \,(\eta_t (\x) = \eta_t (\y) = 1) \ = \ N^2 \,\Phi_t^{2, N} (\u^N (0)). \end{array}
\end{equation}
 Combining~\eqref{eq:phitruncomp-1}--\eqref{eq:phitruncomp-2} and~\eqref{eq:occupation-density-2}--\eqref{eq:occupation-density-3}, we deduce that
 $$ \begin{array}{rcl}
    \var \,(\xi_t (x)) & = & N^2 \,(\Phi_t^{2, N} (\u^N (0)) - (\Phi_t^N (\u^N (0)))^2) \vspace*{4pt} \\
                       & \leq & N^2 \,(\Phi_t^{2, N} (\u^N (0)) - \phi_t^2 (\u^N (0)) + (\phi_t (\u^N (0)))^2 - (\Phi_t^N (\u^N (0)))^2) \vspace*{4pt} \\
                       & \leq & 2N^2 \,P \,(\tau^N \leq t). \end{array} $$
 This, together with Chebyshev's inequality, implies that
 $$ \begin{array}{l}
       P \,(|u_x^N (t) - (\Phi_t^N (\u^N (0)))_x| > \ep) \ = \ P \,(|u_x^N (t) - E \,(u_x^N (t))| > \ep) \vspace*{4pt} \\ \hspace*{50pt} \leq \
         \ep^{-2} \,\var \,(u_x^N (t)) \ = \ (\ep N)^{-2} \,\var \,(\xi_t (x)) \ \leq \ 2 \,\ep^{-2} \ P \,(\tau^N \leq t) \end{array} $$
 showing the result under the symmetry assumption~\eqref{eq:occupation-density-1}.
 To prove the result in the absence of symmetry, it suffices to symmetrize the distribution of the initial configuration~$\eta_0$ without changing the distribution of~$\xi_0$.
 If the distribution of~$\eta_0$ concentrates on configurations with a finite number of individuals, then the symmetrized initial distribution~$\eta_0'$ is given by
 $$ P \,(\eta_0' = \eta') \ = \ \sum_{\eta : \xi = \xi'} \ \prod_{x : \xi' (x) \neq 0} \ \binom{N}{\xi' (x)}^{-1} \ P \,(\eta_0 = \eta) $$
 where~$\xi'$ corresponds to~$\eta'$ and in the sum, $\xi$ corresponds to~$\eta$.
 More general distributions can be symmetrized by first applying the above formula to the finite-dimensional distributions, and then taking a limit.
\end{proof}


\section{Properties of the mean-field equations and proof of Theorems~\ref{th:ode-spread}--\ref{th:expansion-retreat}}
\label{sec:mean-field}

\indent Recall the mean-field equations \eqref{eq:mean-field}:
 $$ u_x' \ = \ \bigg(a u_x^2 + \frac{b}{2M} \,\sum_{y \sim x} \,u_y^2 \bigg)(1 - u_x) - u_x \quad \hbox{for all} \quad x \in \Z. $$
 First we address existence of solutions.
 Clearly, the set~$\mathcal{K} = [0, 1]^{\Z}$ is invariant.
 In addition, since the right-hand side of~\eqref{eq:mean-field} is Lipschitz continuous with respect to the sup-norm~$\|u\|_{\infty} = \sup_x |u_x|$, standard theory guarantees existence and uniqueness of solutions.

\indent From the stability analysis of the single equation~\eqref{eq:smf} and Corollary~\ref{cor:mt}, it follows that
\begin{itemize}
\item if~$r < 4$ then for any~$\u (0) \in \mathcal{K}$, $u_x (t) \to 0$ uniformly in~$x$ as~$t \to \infty$, and \vspace*{4pt}
\item if~$r > 4$ and~$\inf_x u_x (0) > u_-$ then $u_x (t) \to u_+$ uniformly in~$x$ as~$t \to \infty$.
\end{itemize}
 Therefore the zero solution is stable when~$r < 4$, and the positive equilibrium~$u_x \equiv u_+$ is stable in the sup-norm when~$r > 4$.

\indent We now address expansion and retreat, defined in the introduction, which leads to more robust notions of stability.
 We begin with a lemma on \emph{wave fronts}, where the set of wave front is
 $$ W \ := \ \{\u \in \mathcal{K} : u_x \geq u_y \ \hbox{for all} \ x \leq y \}. $$
 In what follows, the property stated in Corollary~\ref{cor:mt} is called \emph{monotonicity}.
\begin{lemma} --
\label{lem:mtpf}
 If~$\u (0) \in W$ then~$\u (t) \in W$ for all~$t > 0$.
\end{lemma}
\begin{proof}
 For~$z \in \Z$, we define the shift map~$\tau_z \u$ on~$\mathcal{K}$ by~$(\tau_z u)_x = u_{x - z}$.
 Then, writing the system~\eqref{eq:mean-field} as~$\u' = F (u)$, we easily check that~$F (\tau_z (\u)) = \tau_z (F (\u))$, from which it follows that
 $$ \v (t) = \tau_z \u(t) \ \hbox{solves~\eqref{eq:mean-field}} \quad \hbox{whenever} \quad \u (t) \ \hbox{solves~\eqref{eq:mean-field}}. $$
 For a wave front~$\u$, one has~$\tau_z \u \geq \u$ for all~$z\geq 0$.
 If~$\u(0)$ is a wave front, by monotonicity, it follows that, for~$z \geq 0$, $\tau_z \u (t) \geq \u (t)$, therefore~$u_{y - z} (t) \geq u_y (t)$ for all~$y$.
 In particular, for all~$x \leq y$ it suffices to take~$z = y - x$ to see that~$u_x (t) \geq u_y (t)$.
\end{proof} \\ \\
 Using~Lemmas~\ref{lem:mfappr} and~\ref{lem:mtpf}, we can prove the first part of Theorem~\ref{th:expansion-retreat}.
\begin{lemma} --
 Expansion and retreat are open conditions.
\end{lemma}
\begin{proof}
 By continuity of the solutions to~\eqref{eq:mean-field} with respect to~$a$ and~$b$, it is enough to show that if expansion occurs for a given value~$u$ then~$u_1 (t_1) > u$ for some~$t_1$, and that if retreat occurs for given
 values~$u_1, u_2$ then~$u_{-1} (t_1) < u_1$ for some~$t_1$. \\
\indent Since~$W$ is invariant according to Lemma~\ref{lem:mtpf} and~$u \,\mathbf{1} (x \leq 0) \in W$, if expansion occurs with~$u_1 (t_0) = u$ then~$\u (t) \geq u \,\mathbf{1} (x \leq 1)$ and so~$\tau_{-1}\u(t_0) \geq \u(0)$.
 Iterating, for integer~$m \geq 1$, we get~$\tau_{-m} \u (m t_0) \geq \u (0)$, which implies that~$u_x (m t_0) \geq u$ for~$x \leq m$. Now,
\begin{itemize}
 \item let~$u (t)$ solve~\eqref{eq:smf} with~$u (0) = u$, \vspace*{5pt}
 \item let~$t_2$ be such that~$u (t_2) = (u + 2u_+) / 3$, \vspace*{5pt}
 \item let~$c, \gamma > 0$ be as in Lemma~\ref{lem:mfappr}, \vspace*{5pt}
 \item let~$t_3$ be such that~$2 e^{-\gamma t_3} = (u_+ - u) / 3$, \vspace*{5pt}
 \item let~$m$ be large enough that~$m > c \max (t_2, t_3)$.
\end{itemize}
 Using Lemma~\ref{lem:mfappr} and monotonicity, we obtain~$u_1 (m t_0) \geq (2u + u_+) / 3 > u$.
 Since the proof for retreat is similar, it is omitted.
\end{proof} \\ \\
 With a bit more work we obtain Theorem \ref{th:ode-spread}. \\ \\
\begin{demo}{Theorem~\ref{th:ode-spread}} --
 In the previous proof, we showed expansion for a given value~$u$ implies~$u_1 (t_1) > u$ and by invariance of~$W$, $u_x (t_1) \geq u_1 (t_1) > u$, for some~$t_1$ and every~$x \leq 1$.
 Let~$\ep := u_1 (t_1) - u > 0$, and let~$\u (0) := u \,\mathbf{1} ([-L, L])$.
 Using this and Lemma~\ref{lem:mfappr}, it follows that
 $$ u_x (t_1) \ \geq \ u + 2 \ep / 3 \quad \hbox{for all} \quad 0 \leq x \leq L + 1 $$
 for all~$L$ large enough, and then by symmetry,
 $$ u_x (t_1) \ \geq \ u + 2 \ep / 3 \quad \hbox{for all} \quad - (L + 1) \leq x \leq L + 1 $$
 Comparing to the increasing solution~$\v (t)$ with~$\v (t_1) = (u + 2 \ep / 3) \,\mathbf{1} (\Z)$ and using again Lemma~\ref{lem:mfappr}, if~$L$ is large enough then additionally~$u_0 (s) \geq u + \ep / 3$ for~$t_1 \leq s \leq 2t_1$.
 Using monotonicity and iterating, we find that for integer~$m \geq 1$,
 $$ \begin{array}{rclcl}
      u_x (m t_1) & \geq & u + 2 \ep / 3 & \hbox{for all} & x \in [- (L + m), L + m] \ \ \hbox{and} \vspace*{4pt} \\
      u_x (s)     & \geq & u + \ep / 3 \ \geq \ u & \hbox{for all} & s \geq (m + 1) \,t_1 \ \ \hbox{and} \ \ x \in [-m, m]. \end{array} $$
 The conclusion of Theorem~\ref{th:ode-spread} is satisfied for~$x_0 = 1$, $\delta = \ep / 3$ and~$c = 1/t_1$.
 Since the proof for retreat is similar, it is omitted.
\end{demo} \\ \\
 To conclude this section, we show in the next two lemmas the two quantitative estimates that complete the proof of Theorem~\ref{th:expansion-retreat}.
\begin{lemma} --
 Expansion occurs for $M=1$ if $a+b/2>4$ and $b>8/9$.
\end{lemma}
\begin{proof}
 Let~$\u (0) = u \,\mathbf{1} (x \leq 0)$, for~$u$ to be determined.
 Since~$W$ is invariant and~$\u (0) \in W$, $\u(t) \in W$ for~$t > 0$ which implies that~$u_{-1} (t) \geq u_0 (t)$ for~$t \geq 0$.
 Therefore, $u_0 (t)$ satisfies the differential inequality
 $$ u_0' \ \geq \ (a + b / 2) \,u_0^2 \,(1 - u_0) - u_0. $$
 If~$a + b / 2 > 4$ and~$u > 1/2$ then~$\liminf_{t \to \infty} u_0 (t) \geq \tilde{u}_+$, the upper equilibrium of~\eqref{eq:smf} for~$\tilde{r} = a + b/2$.
 Since~$\tilde{u}_+ > 1/2$, let~$t_1$ be such that~$u_0 (t) \geq 1/2$ for~$t\geq t_1$, then for~$t \geq t_1$, $u_1 (t)$ satisfies the differential inequality
 $$ u_1' \ \geq \ (au_1^2 + b / 8)(1 - u_1) - u_1. $$
 Denoting the right-hand side~$P (u_1)$, suppose that~$P (u_1) > 0$ for~$u_1 \in [0, 1/2]$.
 Letting~$u (t)$ solve the equation~$u' = P (u)$ with~$u (0) = 0$,
 $$ u_- < 1/2 < u(t_2) < u_+ \quad \hbox{for some} \quad t_2 $$
 where~$u_-,u_+$ are the non-trivial equilibria of~\eqref{eq:smf} for~$r = a + b$.
 Letting~$u = u (t_2)$, $u > 1/2$ and by comparison, $u_1 (t_1 + t_2) \geq u$, which implies expansion. \\
\indent It remains to find conditions that guarantee~$P (u) > 0$ for~$u \in [0,1/2]$.
 If~$a + b / 2 = 4$,
 $$ P (u) \ = \ Q (u) (u - 1/2) \quad \hbox{where} \quad Q (u) := (- au^2 + (a / 2) \,u - b/4). $$
 For any~$u$, provided~$a < 4b$,
 $$ Q (u) \ \leq \ Q (1/4) \ = \ a/8 - b/4 \ < \ 0, $$
 which implies~$P (u) > 0$ for~$u \in [0,1/2)$.
 Moreover, $P$ is non-decreasing with respect to~$a, b$, and~$P (1/2) > 0$ if~$a + b/2 > 4$.
 The intersection of the lines
 $$ a + b/2 = 4 \ \ \hbox{and} \ \ a = 4b \ \ \hbox{is at} \ \ (a, b) = (1/9)(32, 8). $$
 By monotonicity we find that expansion occurs whenever~$a + b/2 > 4$ and~$b > 8/9$.
\end{proof}
\begin{lemma} --
 Retreat occurs for~$M = 1$ if~$a + b \leq 4$ and~$b > 0$.
\end{lemma}
\begin{proof}
 Given~$a, b$ with~$a + b = 4$, we will produce a wave front solution~$\u (t)$ with
 $$ u_x (0) = u^* > u_+ \ \ \hbox{for} \ \ x < 0 \quad \hbox{and} \quad u_x (0) = u_* < u_- \ \ \hbox{for} \ \ x \geq L $$
 such that~$\tau_1 (\u (t_1)) \leq \u (0)$ for some~$t_1$.
 Iterating, for integer~$k \geq 1$, it follows that~$\tau_k (\u (k t_1)) \leq \u (0)$, which means~$u_x (k t_1) \leq u_*$ for~$x \geq L-k$.
 Letting~$k = L + 1$ and using monotonicity, it follows that the solution with~$\u (0) = u^* \,\mathbf{1} (x < 0) + u_* \,\mathbf{1} (x \geq 0)$ has
 $$ u_x ((L + 1) \,t_1) \leq u_* \quad \hbox{for all} \quad x \geq L - (L + 1) = -1, $$
 satisfying the conditions for retreat with~$t_0 = (L + 1) \,t_1$. \\
\indent Let~$f (u) = 4u^2 \,(1 - u) - u$ denote the right-hand side of~\eqref{eq:smf} for~$r = a + b = 4$. Then,
\begin{itemize}
 \item $f (u) \leq 0$ for~$u \in [0, 1]$ with~$f (0) = 0$ and~$f (1) = -1$, \vspace*{4pt}
 \item $f (u)$ has a local maximum at~$u = 1/2$ with~$f (1/2) = 0$ and \vspace*{4pt}
 \item $f (u)$ has a local minimum at some~$u_* \in (0, 1/2)$ with~$f (u_*) < 0$.
\end{itemize}
 Since $M = 1$, the mean-field equations are
 $$ u_x' \ = \ (au_x^2 + (b/2)(u_{x - 1}^2 + u_{x + 1}^2))(1 - u_x) - u_x $$
 For~$\ep, L > 0$ to be determined, let~$u^* = 1/2 + \ep$ and
\begin{itemize}
 \item let~$u_x (0) = u^*$ for all~$x < 0$, \vspace*{4pt}
 \item let~$u_0 (0) = 1/2$, \vspace*{4pt}
 \item let~$u_x (0)$ be such that~$u_{y - 1}^2 + u_{y + 1}^2 = 2 \,u_y^2$ for all~$y \in [0, L)$, which gives
  $$ u_x \ = \ ((1/2)^2 - \ep \,(1 + \ep) \,x)^{1/2}. $$
  where~$L$ is such that if~$u_{L - 1}, u_L$ are decided by the above formula, then~$u_{L - 1} > u_* \geq u_L$, \vspace*{4pt}
 \item let~$u_x (0) = u_*$ for~$x \geq L$.
\end{itemize}
 Note that by choosing~$\ep>0$ small enough, we can make~$u_{L - 1} (0) - u_*$ as small as we like.
 For this initial data, we have the following properties:
\begin{itemize}
 \item $u_x' (0) = f (u_x) = f (u^*) < 0$ for all~$x < 0$, \vspace*{4pt}
 \item $u_0' (0) = f (1/2) = 0$ by construction, \vspace*{4pt}
 \item $u_x' (0) = f (u_x (0)) < 0$ for~$x \in (0, L)$ by construction, \vspace*{4pt}
 \item $u_L' (0) \leq f (u_*) + \delta$ for some~$\delta > 0$ such that~$\delta \to 0$ as $\ep \to 0$, \vspace*{4pt}
 \item $u_x' (0) = f (u_*) < 0$ for all~$x > L$.
\end{itemize}
 Taking~$\ep > 0$ small enough that~$\delta < |f(u^*)| / 2$, we have~$u_{L-1}' (0) < 0$, and then redefining~$u^* = 1/2 + \ep / 2$, we still have~$u_x' (0) = f (u^*) <0$ for~$x < 0$.
 In addition, we have~$u_0' (0) < 0$ since~$u_0'$ is strictly decreasing with~$u_{-1}$.
 Since~$u_x' (0) < 0$ for all~$x$, and uniformly outside a finite set, we have~$u_x' (0) \leq -\delta < 0$ for all~$x$ and some (probably different) $\delta > 0$.
 Since~$\u (t)$ is strictly decreasing at time~$0$, for every~$s \in [0, h]$ for some~$h > 0$, one has~$\u (h) \leq \u (0)$ and, by monotonicity, one easily shows that~$\u (t + s) \leq \u (t)$ for~$t, s \geq 0$.
 Since~$u_x'$ is non-decreasing in~$u_{x \pm 1}$ and is non-decreasing in~$u_x$ for~$u_x \in [u_*, 1/2]$, it follows that for all~$x \in [0, L)$, we have~$u_x' (t) \leq - \delta < 0$ so long as~$u_x \geq u_*$, so
 there is~$t_1$ so that~$u_x (t_1) \leq u_*$ for all~$x \in [0, L)$ and in particular, $\tau_1 (\u (t_1)) \leq \u (0)$.
\end{proof}


\section{Block construction for expansion and proof of Theorem~\ref{th:expansion}}
\label{sec:expansion}

\indent In this section, we combine the behavior of the mean-field model with the estimates established in Section~\ref{sec:occupation} to prove Theorem~\ref{th:expansion}.  
 To do so, we use the comparison technique described in~\cite[section~4]{durrett_1995}, also known as block construction, to couple the~$N$-patch model to an oriented site percolation model on the directed graph~$\mathcal H$ with vertex set
 $$ H \ := \ \{(z, n) \in \Z \times \Z_+ : z + n \ \hbox{is even} \} $$
 and in which there is an oriented edge
 $$ (z, n) \to (z', n') \quad \hbox{if and only if} \quad z' = z \pm 1 \ \ \hbox{and} \ \ n' = n + 1. $$
 See Durrett~\cite{durrett_1984} for a definition and review of oriented percolation.
 In the vocabulary of~\cite{durrett_1995}, suppose we are given a~$k$-dependent percolation model with density at least, i.e., in which sites are open with probability at least, $1 - \gamma$, and let~$\mathcal C_0$ denote the cluster containing
 the origin, i.e., the set of sites that can be reached from the origin by an open path.
 We recall the key ingredients from the reference, whose statements have been specialized somewhat to fit our context.
\begin{theorem}[\cite{durrett_1995}, Theorem 4.1] --
\label{th:st-flour-1}
 We have~$P \,(|\mathcal C_0| < \infty) \to 0$ as~$\gamma \to 0$.
\end{theorem}
 Given an initial configuration~$W_0 \subset \{z : (z, 0) \in H \}$, let
\begin{equation}
\label{eq:wet}
  W_n \ := \ \{z : \hbox{there is an open path} \ (y, 0) \to (z, n) \ \hbox{for some} \ y \in W_0 \}.
\end{equation}
 The next result will help us to get a stationary distribution.
\begin{theorem}[\cite{durrett_1995}, Theorem 4.2 with~$p = 1$] --
\label{th:st-flour-2}
 Let~$W_0 = 2 \Z$. Then,
 $$ \begin{array}{l} \liminf_{n \to \infty} \,P \,(0 \in W_{2n}) > 0 \quad \hbox{whenever} \quad \gamma > 0 \ \hbox{is small enough}. \end{array} $$
\end{theorem}
 Suppose we have a collection of configurations~$\Lambda$ such that the truth of the statement~$\eta \in \Lambda$ depends only on the values~$\{\eta (\x) : \pi(\x) \in [-L, L] \}$ and let
\begin{equation}
\label{eq:occupied}
  X_n \ := \ \{m \in \Z : (m, n) \in H \ \hbox{and} \ \sigma^{-2mL} \,\eta_{n T} \in \Lambda \}
\end{equation}
 where~$\sigma^y$ for~$y \in \Z$ is defined by $\sigma^y \,\eta (x, j) = \eta(x - y, j)$, and where~$L$ and~$T$ give the appropriate space and time scales for the block construction.
 The sets in~\eqref{eq:wet}--\eqref{eq:occupied} are usually referred to as the set of \emph{wet} sites and the set of \emph{occupied} sites at level~$n$.
 The next theorem gives a sufficient condition for the set of occupied sites to dominates stochastically the set of wet sites.
\begin{theorem}[\cite{durrett_1995}, Theorem 4.3] --
\label{th:st-flour-3}
 Suppose~$(\eta_t)_{t \geq 0}$ is a translation-invariant finite range process and for each~$\eta_0 \in \Lambda$ there is an event~$G (\eta_0)$ such that the following comparison assumptions hold:
\begin{itemize}
 \item the event~$G (\eta_0)$ is measurable with respect to the graphical representation of the process in the finite space-time box~$[-kL, kL] \times [0, kT]$ where~$k$ is the range of dependency of the percolation process introduced above,\vspace*{4pt}
 \item we have the inclusion~$\{\eta_0 \in \Lambda \} \cap G (\eta_0) \subset \{\sigma^{2L} \,\eta_T \in \Lambda \} \cap \{\sigma^{-2L} \,\eta_T \in \Lambda \}$ and \vspace*{4pt}
 \item the scale parameters~$L$ and~$T$ can be chosen such that~$P \,(G (\eta_0)) \geq 1 - \gamma$.
\end{itemize}
 Then, the process~$X_n$ dominates~$W_n$ provided~$W_0 \subset X_0$.
\end{theorem}
 With the previous three theorems and the results from the previous sections in hands, we are now ready to prove Theorem~\ref{th:expansion}. \\ \\
\begin{demo}{Theorem~\ref{th:expansion}} --
 Suppose~$a$ and~$b$ are such that expansion occurs in the mean-field equations and let~$u, L, x_0, \delta, c > 0$ with~$u_- < u < u_+$ be as in the statement of Theorem~\ref{th:ode-spread}. Then,
 $$ u_x (0) \geq u \ \ \hbox{for all} \ \ |x| \leq L \quad \hbox{implies that} \quad u_x (t) \geq u + \delta \ \ \hbox{for all} \ \ |x| \leq ct - x_0 $$
 so that, letting~$T := (3 L + x_0) / c$, we get
\begin{equation}
\label{eq:expansion-1}
  cT - x_0 = 3L \quad \hbox{and} \quad u_x (T) \geq u + \delta \ \ \hbox{for all} \ \ |x| \leq 3L.
\end{equation}
 Now, for an interval~$[-K, K] \subset \Z$, let~$^K\eta_t$ denote the restriction of the~$N$-patch model to this interval with lower boundary values, which is the process constructed from~$\eta_0$ by ignoring the Poisson processes
 $$ a_n (\x, \y, \z), \ b_n (\x, \y, \z), \ d_n (\x) \ \ \hbox{such that} \ \ \pi (\x), \pi (\y) \ \hbox{or} \ \pi (\z) \notin [-K, K]. $$
 In the same way, define the restriction of the~$N$-dual and let~$^Ku_x (t)$ denote the restriction with lower boundary values as given by Definition \ref{def:restriction}.
 Also, let
 $$ \Lambda \ := \ \{\eta : \xi (x) \geq uN \ \hbox{for all} \ |x| \leq L \}, $$
 where~$\xi$ is the mesoscopic configuration corresponding to~$\eta$.
 That is, $\Lambda$ is the set of microscopic configurations with at least~$uN$ occupied sites at each location in the interval~$[-L, L]$.
 One easily verifies that Lemma~\ref{lem:collision-estimate} and Theorem~\ref{th:occupation-density} hold for the restrictions defined above.
 Using in addition Lemma~\ref{lem:block-estimate-fixed} and~\eqref{eq:expansion-1}, we deduce that, for all~$\gamma > 0$, there is~$K$ large such that
\begin{equation}
\label{eq:expansion-2}
 \begin{array}{l}
  P \,(\sigma^{2L} \,^K\eta_T \in \Lambda \ \hbox{and} \ \sigma^{- 2L} \,^K\eta_T \in \Lambda \ | \ ^K\eta_0 \in \Lambda) \vspace*{4pt} \\ \hspace*{40pt} = \
  P \,(^K\xi_T (x) \geq uN \ \hbox{for all} \ |x| \leq 3L \ | \ ^K\eta_0 \in \Lambda) \vspace*{4pt} \\ \hspace*{40pt} \geq \
  P \,(^K\xi_T (x) \geq N \,(^Ku_x (T) - \delta / 2) \ \hbox{for all} \ |x| \leq 3L \ | \ ^K\eta_0 \in \Lambda) \vspace*{4pt} \\ \hspace*{40pt} \geq \
  1 - (3L + 1) \times 2 \,(\delta / 2)^{-2} \,P \,(\tau^N \leq T) \ \geq \ 1 - \gamma \end{array}
\end{equation}
 for all~$N$ large enough.
 This shows that there exists an event~$G (\eta_0)$ that satisfies the last two comparison assumptions in the statement of Theorem~\ref{th:st-flour-3}.
 Since the estimates in~\eqref{eq:expansion-2} holds for the process where births outside the interval~$[-K, K]$ are suppressed, this event can also be chosen to be measurable with respect to the
 graphical representation in the corresponding space-time box, showing that the first comparison assumption is satisfied as well.
 In particular, Theorem~\ref{th:st-flour-3} can be applied which, together with Theorem~\ref{th:st-flour-1}, implies that there exists~$L$ such that
 $$ \begin{array}{l} \lim_{N \to \infty} \,P \,(\xi_t (x) = \0 \ \hbox{for some} \ t > 0 \ | \ \eta_0 \in \Lambda) \ = \ 0 \end{array} $$
 By attractiveness, the same holds when starting with all patches in~$[-L, L]$ fully occupied.
 To also prove the existence of a nontrivial stationary distribution, we consider the process starting from the all occupied configuration.
 Due again to attractiveness, the distribution of~$\eta_t$ is stochastically decreasing in time and by the general result~\cite[Chapter III, Theorem 2.3]{liggett_1985}, the distribution of~$\eta_t$ converges to a translation invariant,
 and on each patch also permutation invariant, stationary distribution~$\nu$.
 Combining Theorems~\ref{th:st-flour-2}--\ref{th:st-flour-3}, we get
 $$ \nu (\{\eta : \eta(\x) = 1 \}) > 0 \quad \hbox{for some} \quad \x \in D_N $$
 but by translation and permutation invariance, the same is true for all~$\x \in D_N$.
 In particular, $\nu$ does not concentrate on the all-zero configuration and the proof is complete.
\end{demo}


\section{Block construction for retreat and proof of Theorem~\ref{th:retreat}}
\label{sec:retreat}

\indent This section is devoted to the proof of Theorem~\ref{th:retreat}.
 That is, we show that if the inner and outer birth rates~$a$ and~$b$ are such that retreat occurs in the mean-field equations, then for~$N$ large enough the distribution of the~$N$-patch stochastic model
 converges weakly to the point mass on the all-zero configuration.
 The strategy is as follows.
\begin{enumerate}
 \item Wait until~$\xi_t (x) / N$ is uniformly small for all~$|x| \leq L$, for some large enough~$L$. \vspace*{4pt}
 \item Using a block construction, attempt to bootstrap from this configuration to percolation of low density regions in space and time. \vspace*{4pt}
 \item If the construction fails, then wait until a low density region appears and try again. \vspace*{4pt}
 \item If the construction succeeds, then show that there is a completely vacant zone that grows asymptotically linearly in time.
\end{enumerate}
 Since for fixed~$L$ and~$N$ the number of sites~$\x$ such that~$|\pi (\x)| \leq L$ is finite, for each~$h > 0$, there is a positive probability that in the time interval~$[t, t + h]$ every site becomes
 vacant without any new births onto those sites occurring.
 Thus, although it may take a long time, for any stopping time~$t$, the first time~$s$ such that~$\xi_{t + s} (x) = 0$ for all~$|x| \leq L$ has an exponential tail and, in particular, is finite
 with probability one. \\
\indent The idea of retrying the construction until it succeeds appears in \cite{neuhauser_1994} and other references.
 If the construction succeeds with positive probability and fails within finite time when it does not succeed, then we need only to try a geometrically distributed number of times until it succeeds.
 Standard techniques (see for example the contour arguments in the appendix of \cite{durrett_1995}) guarantee that for a finitely dependent percolation model with density~$\geq 1 - \gamma$, if
 $$ |W_0| < \infty \quad \hbox{and} \quad P \,(W_n \neq \varnothing \ \hbox{for all} \ n) < 1 $$
 then conditioned on the extinction event~$\{W_n = \varnothing \ \hbox{for some} \ n \}$, the least such value of~$n$ has an exponential tail and in particular, is finite with probability one. \\
\indent Our first step is the following improvement of Theorem~\ref{th:ode-spread}.
\begin{lemma}
\label{lem:low-density} --
 If retreat occurs, then there exist~$u, L, c > 0$ with~$0 < u < u_-$ and, for all~$\ep > 0$, there exists an~$x_0 (\ep) > 0$ such that
 $$ u_x (0) \leq u \ \ \hbox{for all} \ \ |x| \leq L \quad \hbox{implies that} \quad u_x (t) \leq \ep \ \ \hbox{for all} \ \ |x| \leq ct - x_0 (\ep). $$
\end{lemma}
\begin{proof}
 Using Theorem~\ref{th:ode-spread}, we have~$u, L, x_1, \delta, c > 0$ with~$0 < u < u_-$ so that
\begin{equation}
\label{eq:low-density-1}
  u_x (0) \leq u \ \ \hbox{for all} \ \ |x| \leq L \quad \hbox{implies that} \quad u_x (t) \leq u - \delta \ \ \hbox{for all} \ \ |x| \leq ct - x_1.
\end{equation}
 Let~$v (t)$ be the solution to~\eqref{eq:smf} with~$v (0) = u - \delta$, and let~$c_0, \gamma_0$ be as in Lemma~\ref{lem:mfappr}.
 Since~$v (t)$ converges to zero as~$t \to \infty$,
 $$ \max \,(v (t_0), 2 e^{- \gamma_0 t_0}) \ \leq \ \ep/2 \quad \hbox{for some} \quad t_0 > 0. $$
 Let~$x_0 := c t_0 + c_0 t_0 + x_1$.
 Then, according to~\eqref{eq:low-density-1}, for fixed $t$,
 $$ \begin{array}{rcl}
      |x| \leq ct - x_0 \ \hbox{and} \ |y - x| \leq c_0 t_0 & \hbox{imply that} & |y| \leq ct + c_0 t_0 - x_0 = c \,(t - t_0) - x_1 \vspace*{4pt} \\
                                                            & \hbox{imply that} & u_y (t - t_0) \leq u. \end{array} $$
 In particular, letting~$\v (s)$ be the solution to~\eqref{eq:mean-field} with~$v_x (s) = v (s - (t - t_0))$ for all~$x$, using Lemma~\ref{lem:mfappr} and monotonicity, we conclude that
 $$ u_x (t) \ \leq \ v_x (t) + \ep/2 \ \leq \ \epsilon $$
 as desired.
\end{proof} \\ \\
 Starting from a low density region, we create a larger, completely vacant region in two steps in the following manner.
 First, we show that after some time the density becomes low and then remains low for a while.
 Then, we use the maintenance of that low density to show that the region becomes completely vacant.
 We begin with the first of these two steps.
 For~$\alpha_1, \alpha_2, \alpha_3 > 0$ to be determined in a moment, we let
 $$ T_0 := \alpha_1 \,\log N \quad \hbox{and} \quad L := \alpha_2 \,\log N \quad \hbox{and} \quad K := \alpha_3 \,\log N. $$
 Let~$u, L, x_0 (\ep), c > 0$ be as in Lemma~\ref{lem:low-density} and, in the terminology of Section \ref{sec:expansion}, define
 $$ \Lambda = \{\eta : \xi(x) \leq uN \ \hbox{for all} \ |x| \leq L \}$$
 where $\xi$ is the mesoscopic configuration corresponding to $\eta$.
 We want to make the density small on some region, and then exploit this fact to show that the population goes to zero in that region.
 Since we need an upper bound, define the~$K$-restrictions using the upper boundary values~$^Ku_y(t) \equiv 1$ for~$|y| > K$, so that~$^Ku_x(t) \geq u_x (t)$ for all~$x, t$. 
 For $0 < \ep < \min \,(u, (a + b)^{-1})$ to be chosen later, define
 $$ \Lambda_0 = \{\eta : \xi_t (x) \leq \ep N \ \hbox{for all} \ |x| \leq 4L \} \quad \hbox{and} \quad G_0 (\eta_0) = \{{}^K\eta_t \in \Lambda_0 \ \hbox{for all} \ t \in [T_0/2, T_0] \} $$
 Then, we have the following estimate.
\begin{proposition} --
\label{propo:lambda0}
 For each~$\ep > 0$ and small enough~$\alpha_1 > 0$, there are~$\alpha_2, \alpha_3 > 0$ and~$C_1, \gamma_1 > 0$ so that, for all~$N$ large enough,
 $$ P \,(G_0 (\eta_0) \ | \ \eta_0 \in \Lambda) \ \geq \ 1 - C_1 \,N^{- \gamma_1 / 2}. $$
\end{proposition}
\begin{proof}
 Let~$u, L_0, x_0 (\ep), c$ be as above and fix~$x_0 = x_0 (\ep/4)$ so that
 $$ u_x (0) \leq u \ \ \hbox{for all} \ \ |x| \leq L_0 \quad \hbox{implies that} \quad u_x (t) \leq \ep/4 \ \ \hbox{for all} \ \ |x| \leq ct-x_0. $$
 Monotonicity implies that if~$L \geq L_0$, then
 $$ u_x (0) \leq u \ \ \hbox{for all} \ \ |x| \leq L \quad \hbox{implies that} \quad u_x (t) \leq \ep/4 \ \ \hbox{for all} \ \ |x| \leq ct + L - (L_0 + x_0). $$
 Fixing~$\alpha_2 > 0$ small enough that~$\alpha_1 > 6 \,\alpha_2 / c$, we have
 $$ c (T_0 / 2) + L - (L_0 + x_0) \ \geq \ 4L \quad \hbox{when~$N$ is large enough}. $$
 Then, take~$\alpha_3$ large enough so that, by Lemma~\ref{lem:block-estimate-growing},
 $$ ^Ku_x (t) \leq u_x (t) + N^{-1} \quad \hbox{for every} \quad |x| \leq 4L \ \ \hbox{and} \ \ t \in [0, T_0]. $$
 Now, combining Lemma~\ref{lem:collision-estimate} and Theorem~\ref{th:occupation-density}, we get
 $$ \begin{array}{l}
       P({}^K\xi_t(x)/N \leq {}^Ku_x(t) + \ep/4) \ \geq \ 1 - 2 \,(\ep/4)^{-2} \,P \,(\tau^N \leq t) \vspace*{4pt} \\ \hspace*{25pt} \geq \
       1 - 2 \,(\ep/4)^{-2} \,(2 \,e^{2(a + b) t} + 1) \,N^{-1/3} \ \geq \ 1 - 2 \,(\ep/4)^{-2} \,(2 \,N^{2 (a + b) \,\alpha_1} + 1) \,N^{-1/3}. \end{array} $$
 for fixed~$|x| \leq K$ and~$t \in [T_0/2, T_0]$.
 Letting~$\gamma_1 := 1/3 - 2 (a + b) \,\alpha_1$, which is positive for~$\alpha_1$ small enough, and using the above observations, we deduce that
 $$ P \,({}^K\xi_t (x) / N \leq 3 \ep/4) \ \geq \ 1 - 6 \,(\ep/4)^{-2} \,N^{- \gamma_1} \ = \ 1 - (96 / \ep^2) \,N^{-\gamma_1}. $$
 Suffering a factor of~$8L + 1 = 8 \alpha_2 \log N + 1$ in the estimate, this bound is uniform over~$|x| \leq 4L$.
 To make it also uniform in time, let~$h = (\ep/8)( a + b + 1)^{-1}$.
 Each transition in the process increases or decreases the number of occupied sites by at most one.
 Moreover, for each~$x$, the total rate of Poisson point processes affecting sites in patch~$x$ is at most~$(a + b + 1) \,N$.
 Thus, for each~$j$,
 $$ \begin{array}{l} \sup_{s \in [j h, (j + 1) h]} \,|\xi_s (x) - \xi_{jh} (x)| \end{array} $$
 is dominated by~$X = \poisson (\mu)$ where~$\mu = (a + b + 1) \,Nh = \epsilon N/8$.
 A standard large deviations bound gives~$P \,(X > 2 \mu) \leq e^{-\mu/4}$.
 Summing over~$|x| \leq 4L$ and~$T_0/(2h) \leq j < T_0/h$,
 $$ \begin{array}{l}
       P \,({}^K\xi_t (x) / N \leq \ep \ \hbox{for all} \ |x| \leq 4L \ \hbox{and} \ t \in [T_0/2, T_0]) \vspace*{4pt} \\ \hspace*{40pt} \geq \
       1 - (8 \alpha_2 \,\log N + 1)(\alpha_1 / (2h) \log N)(96 / \ep^2) \,N^{-\gamma_1} \vspace*{4pt} \\ \hspace*{80pt} - \
       (8 \alpha_2 \,\log N + 1)(\alpha_1 / (2h) \log N) \,e^{-(a + b + 1) \,N h/4} \end{array} $$
 for all~$N$ large enough, and the result follows.
\end{proof} \\ \\
 If we only wanted percolation of space-time boxes with low density then the good event~$G_0$ would suffice.
 In order to take low density to no density, however, we need a bit more work, and a couple more events.
 Let~$T := 2 \,\log N$ and define
 $$ G_1 (\eta_0) = \{{}^K\eta_t \in \Lambda_0 \ \hbox{for all} \ t \in [T_0/2, T + T_0/2] \} $$
 Since~$\ep \leq u$, we have~$\Lambda_0 \subset \Lambda$ therefore, iterating at most~$4 / \alpha_1 + 1$ times the proof of Proposition~\ref{propo:lambda0} by shifting things upwards in time by~$T_0/2$ at each iteration,
 we find that
 $$ P \,(G_1(\eta_0) \ | \ \eta_0 \in \Lambda) \ \geq \ 1 - (4 / \alpha_1 + 1) \,C_1 \,N^{-\gamma_1 / 2}. $$
 Now, define the new collection of configurations
 $$ \Lambda_1 = \{\eta : \xi (x) = 0 \ \hbox{for all} \ |x| \leq 3L \} $$
 and fix~$\eta \in \Lambda_0$.
 For~$t \in [T_0/2,T]$, we define the following processes which are coupled by being constructed from the same graphical representation:
\begin{itemize}
 \item let~$\eta_t^2$ be~$^K\eta_t$ as defined above, started from~$\eta$ at time~$T_0/2$, \vspace*{4pt}
 \item let~$\eta_t^0$ be~$^{4L}\eta_t$ started from~$\eta$ at time~$T_0/2$ with lower boundary values, \vspace*{4pt}
 \item let~$\eta_t^1 = \eta_t^2 - \eta_t^0 \geq 0$ be the occupied sites in~$\eta_t^2$ that are vacant in~$\eta_t^0$.
\end{itemize}
 By construction, $\eta_t^1$ is the set of occupied sites~$\x$ with~$|\pi(\x)| \leq 4L$ that would not be occupied without the help of an \emph{outside} site, i.e., an occupied site~$\y$ such that~$|\pi (\y)| > 4L$. Let
 $$ G_{2, 0} (\eta) = \{\xi_{T + T_0/2}^0 \in \Lambda_1 \} \quad \hbox{and} \quad G_{2, 1} (\eta) = \{\xi_t^1 \in \Lambda_1 \ \hbox{for all} \ t \in [T_0/2, T + T_0/2] \} $$
 where~$\xi_t^i (x) := \card \{\x : \pi (\x) = x \ \hbox{and} \ \eta_t^i (\x) = 1 \}$ for~$i = 0, 1$, and define
 $$ G_2 (\eta) = G_{2, 0} (\eta) \cap G_{2, 1} (\eta). $$
 Now, we say that there is a type~2 active path~$(\x_0, t_0) \to_2 (\x, t)$ if there are
 $$ \x_0, \x_1, \ldots, \x_{n - 1} = \x \in D_N \quad \hbox{and} \quad t_0 < t_1 < \cdots < t_n = t $$
 such that the following two conditions hold:
\begin{itemize}
 \item for~$i = 0, 1, \ldots, n - 1$, we have~$\eta_s^2 (\x_i) = 1$ for all~$s \in [t_i, t_{i + 1}]$ and \vspace*{3pt}
 \item for~$i = 1, 2, \ldots, n - 1$, there is a birth event at time~$t_i$ where the offspring is sent to~$\x_i$ and where~$\x_{i - 1}$ is one of the two parents' locations.
\end{itemize}
 This active path is said to cross the space-time rectangle~$R = [A_1, A_2]\times [S_1, S_2]$ when
 $$ S_1 \leq t_0 \leq t \leq S_2 \quad \hbox{and} \quad \min \,(\pi (\x_0), \pi (\x)) < A_1 \quad \hbox{and} \quad \max \,(\pi (\x_0), \pi (\x)) > A_2. $$
 Then, we define the event~$G_3 (\eta)$ and the set~$B$ as
 $$ \begin{array}{rcl}
      G_3 (\eta) & = & \{\hbox{no active path crosses} \ [-3L, 3L] \times [T_0/2, T + T_0/2] \} \vspace*{4pt} \\
               B & = & \{x \in \Z : |x| \leq 3L \ \hbox{and} \ \xi_t^1 (x) > 0 \ \hbox{for some} \ t \in [T_0/2, T + T_0/2] \}. \end{array} $$
 By definition, we have~$G_{2, 1} (\eta) = \{B = \varnothing \}$.
 Moreover, since the distance from the offspring to each of its parents is at most~$M$, if the set~$B$ has a gap of size~$M$, i.e.,
 $$ B \cap \{x, x + 1, \ldots, x + M \} = \varnothing \quad \hbox{for some} \quad \{x, x + 1, \ldots, x + M \} \subset [-3L, 3L] $$
 then~$G_3 (\eta)$ holds.
 In particular, for~$N$ large enough, $G_{2, 1} (\eta) \subset G_3 (\eta)$.
 Finally, define
 $$ \begin{array}{l} G (\eta_0) = \bigcup_{\eta \in \Lambda_1} \,G_1 (\eta_0) \cap \{^K\eta_{T_0/2} = \eta \} \cap G_2 (\eta) \cap G_3 (\eta) \end{array} $$
 where, after choosing a set of representatives, the union is really over a finite set, as~$^K\eta_t$ depends only on sites~$\x$ such that~$|\pi (\x)| \leq K$.
 One verifies that~$G (\eta_0)$ is measurable with respect to the graphical representation in the space-time rectangle
 $$ [-K, K] \times [0,JT] \quad \hbox{where} \quad J = 1 + \alpha_1/4. $$
 In addition, whenever
 $$ P \,(G_1 (\eta_0) \ | \ \eta_0 \in \Lambda) \geq 1 - \ep_1 \quad \hbox{and} \quad P \,(G_2 (\eta) \ | \ \eta \in \Lambda_1) \geq 1 - \ep_2$$
 after conditioning and reassembling and noting $G_2 \subset G_3$, we find
 $$ P \,(G (\eta_0) \ | \ \eta_0 \in \Lambda) \geq 1 - \ep_1-\ep_2. $$
 Since~$^K\eta \geq \eta$ and using~$G_1$, on~$G (\eta_0)$ we have~$\sigma^{\pm 2L} \eta_T \in \Lambda$ as required. \\
\indent To deduce the existence of vacant regions we pick apart Theorem \ref{th:st-flour-3} somewhat.
 Given~$\eta_{nT}$, we define the event~$G^{(z, n)} (\eta_{nT})$ corresponding to~$G (\eta_0)$ but on the rectangle
 $$ [-K + 2zL, K + 2zL] \times [nT, nT + JT] $$
 and similarly for~$G_1, G_2$ and~$G_3$.
 Recall that~$X_n = \{z \in 2 \Z + n : \sigma^{-2zL} \eta_{nT} \in \Lambda \}$.
 Letting~$Y_0 = X_0$, we define inductively the sequence
 $$ Y_{n + 1} = \{z \in 2 \Z + n : (z \pm 1, n) \in Y_n \ \hbox{and} \ G^{(z \pm 1,n)} (\eta_{nT}) \ \hbox{holds} \} $$
 where the~$\pm$ denotes or and not and. 
 We note that the proof of Theorem~\ref{th:st-flour-3} given in~\cite{durrett_1995} proceeds by demonstrating the stated properties hold for the collection~$(Y_n)_{n\geq 0}$, and then trivially deducing the same properties for~$(X_n)_{n\geq 0}$,
 so we may as well use~$Y_n$.
 For~$n > 0$, define
 $$ \ell_n \ := \ \inf Y_n \quad \hbox{and} \quad r_n \ := \ \sup Y_n $$
 where we use the convention~$\inf (\varnothing) = \infty$ and~$\sup (\varnothing) = - \infty$.
\begin{lemma} --
\label{lem:vacancy}
 For all~$n > 0$, we have
\begin{equation}
\label{eq:vacancy1}
 \xi_{nT + T_0/2} (x) = 0 \quad \hbox{for all} \quad x \in [(2 \ell_n - 1) \,L, (2 r_n + 1) \,L]
\end{equation}
 and for all~$n > 1$ and~$t \in [(n - 1) \,T + T_0/2, n T + T_0/2]$, we have
\begin{equation}
\label{eq:vacancy2}
 \xi_t (x) = 0 \quad \hbox{for all} \quad x \in [(2 \ell_n + 5) \,L, (2 r_n - 5) \,L].
\end{equation}
\end{lemma}
\begin{proof}
 The proof is by induction on $n$.
 By construction, we have~$Y_0 = \{0 \}$.
 Moreover,
 $$ G (\eta_0) = G^{(0,0)} (\eta_0) \ \hbox{holds} \ \ \hbox{and} \ \ \ell_1 = - 1 \ \ \hbox{and} \ \ r_1 = 1 \quad \hbox{whenever} \quad Y_1 = \{0 \} $$
 and one verifies from the definition of~$G$ that~\eqref{eq:vacancy1} holds.
 Therefore, suppose that~\eqref{eq:vacancy1} holds for some~$n$.
 We show that if~$z_1 \leq z_2 \in Y_n$ and~$G^{(z_1, n)}$ and~$G^{(z_2, n)}$ hold then
\begin{equation}
\label{eq:vacancy-proof1}
  \xi_{(n + 1) T + T_0/2} (x) = 0 \quad \hbox{for all} \quad x \in [(2 z_1 - 3) \,L, (2 z_2 + 3) \,L]
\end{equation}
 and for all~$t \in [n T + T_0/2, (n + 1) T + T_0/2]$,
\begin{equation}
\label{eq:vacancy-proof2}
  \xi_t (x) = 0 \quad \hbox{for all} \quad x \in [(2 z_1 + 3) \,L, (2 z_2 - 3) \,L].
\end{equation}
 Since this is true for one of the four combinations~$z_1 = \ell_{n + 1} \pm 1$, $z_2 = r_{n + 1} \pm 1$, the result follows.
 Now, by assumption and since~$\ell_n \leq z_1 \leq z_2 \leq r_n$, we have
 $$ \xi_{n T + T_0/2} (x) =0 \quad \hbox{for all} \quad x \in [(2 z_1 - 1) \,L, (2 z_2 + 1) \,L]. $$
 Also, by~$G_{2, 0}^{(z_1, n)}$ and~$G_{2, 0}^{(z_2, n)}$, we have
 $$ \xi_{(n + 1) T + T_0/2} (x) = 0 \quad \hbox{for all~$x$ such that} \quad |x - 2 z_i L| \leq 3L \quad \hbox{for} \ i = 1, 2. $$
 If~$z_1 \geq z_2 - 3$ then~\eqref{eq:vacancy-proof1} holds, and \eqref{eq:vacancy-proof2} is vacuous, so we are done.
 Otherwise, if we had
 $$ \xi_t (x) > 0 \quad \hbox{for some} \quad (x, t) \in ((2 z_1 + 3) \,L, (2 z_2 - 3) \,L) \times [n T + T_0/2, (n + 1) T + T_0/2] $$
 there would be an active path crossing~$[(2 z_i - 3) \,L, (2 z_i + 3) \,L]$ for either~$i = 1$ or~$i = 2$, which is ruled out by the event~$G_3^{(z_i, n)}$.
\end{proof} \\ \\
 By Lemma~\ref{lem:vacancy}, to conclude that percolation implies weak convergence of the distribution of~$\xi_t$ to the all-zero configuration, it suffices to show that if percolation occurs then~$\ell_n \to -\infty$ and~$r_n \to \infty$.
 Since nothing is lost by increasing~$\xi_0$ to the value~$uN$ on~$[-L, L]$ and setting it equal to one elsewhere, we may assume~$\xi_0$ is symmetric about reflection across~$x = 0$.
 In this case, the distributions of the left and right edges~$\ell_n$ and~$r_n$ can be deduced from one another by symmetry, so it suffices to show that either~$\ell_n \to -\infty$ or~$r_n \to \infty$ when percolation occurs.
 Suppose
 $$ |l_n| \leq r \quad \hbox{and} \quad |r_n| \leq r \quad \hbox{for all~$n \in \N$ and some (random) $r < \infty$}. $$
 Then, for each~$n$, since~$Y_n \subset [-r, r]$, the event
 $$ E_n := \{G^{(z, n)} (\eta_{nT}) \ \hbox{does not hold for any} \ z \in [-r, r] \} $$
 has probability~$P \,(E_n) \geq \ep (r) > 0$.
 Since the percolation model is~$k$-dependent for some~$k < \infty$, for~$n_m = (k + 1) m$ the events~$(E_{n_m})_{m \geq 0}$ are independent so, almost surely, after a geometric number of attempts, $E_{n_m}$ occurs and so there is no percolation. \\
\indent Since, as noted above, $G_{2, 1} \subset G_3$ for all~$N$ large enough, it remains to estimate~$P \,(G_2)$, which is done in the next lemma.
\begin{lemma} --
 There are~$C, \gamma > 0$ so that~$P \,(G_2) \geq 1 - CN^{-\gamma}$.
\end{lemma}
\begin{proof}
 Since~$G_2$ is the intersection of the two events~$G_{2, 0}$ and~$G_{2, 1}$, it suffices to show the lower bound for each of these two events. \vspace*{5pt} \\
{\bf The event~$G_{2, 0}$} --
 Let~$S_t^0 = \sum_{|x| \leq 4L} \,\xi_t^0$.
 Then~$\{S_T^0 = 0 \} \subset G_{2, 0}$ and
 $$ S_t^0 \preceq Z_t \quad \hbox{for all} \quad t \in [T_0/2, T + T_0/2] \quad \hbox{on the event} \quad G_1 (\eta_0) $$
 where~$Z_t$ is a branching process with~$Z_{T_0 / 2} = S_{T_0/2}^0 \leq (8L + 1) \,\ep N$ in which each particle dies at rate one and produces single offspring at rate $(a + b) \,\ep < 1$. Therefore,
 $$ \begin{array}{rcl}
      P \,(G_{2, 0}) & \geq & P \,(S_{T + T_0/2}^0 = 0) \ \geq \ P \,(Z_{T + T_0/2} = 0) \ = \ 1 - P \,(Z_{T + T_0/2} \geq 1) \vspace*{4pt} \\
                     & \geq & 1 - E \,Z_{T + T_0/2} \ \geq \ 1 - (8L + 1) \,\ep N \,\exp (((a + b) \,\ep - 1) \,T) \vspace*{4pt} \\
                     & \geq & 1 - (8 \,\alpha_2 \,\log N + 1) \,\ep N^{- 1 + 2 \,(a + b) \ep} \end{array} $$
 since~$L = \alpha_2 \,\log N$ and~$T = 2 \log N$.
 In particular, for~$\ep > 0$ small so that~$2 \,(a + b) \,\ep < 1$, the probability of the event~$G_{2, 0}$ goes to one exponentially with~$N$, as desired. \vspace*{5pt} \\
{\bf The event~$G_{2, 1}$} --
 We start with the following two observations:
\begin{itemize}
 \item Outside sites, i.e., sites~$\y$ with~$|\pi (\y)| > 4L$, lead to occupied sites inside~$[-4L, 4L]$ at combined rate at most~$2M (a + b) N$, and at locations~$\x$ with~$4L - M < |\pi (\x)| \leq 4L$. \vspace*{4pt}
 \item Each of these, in turn, leads to occupied sites inside~$[-4L, 4L]$ at rate at most~$\lambda = (a + b) \,\epsilon$ and dies at rate one, and each offspring is at most~$M$ patches distant from its parents.
\end{itemize}
 We call the sites that become occupied in the first item \emph{first generation}, and the ones in the second item the \emph{descendants}.
 In order for~$G_{2, 1}$ to occur
 it suffices that the subtree of descendants of each first generation site have depth at most~$L/M - 1$.
 To estimate the probability of this event, we thus look at the number of first generation sites and the depth of a typical subtree. \vspace*{5pt} \\
{\bf Size of the first generation} -- The size~$X$ of the first generation satisfies
 $$ \begin{array}{rcl}
      P \,(X > 4M (a + b) \,NT) & \leq & P \,(\poisson (2M (a + b) \,NT) > 4M (a + b) \,NT) \vspace*{4pt} \\
                                & \leq & \exp \,(- M (a + b) \,NT/2). \end{array} $$
{\bf Depth of a subtree} -- Letting~$\mu := \lambda \,(1 + \lambda)^{-1}$ where~$\lambda = (a + b) \,\ep$, the descendant subtree of each first generation site is dominated by a Galton-Watson tree with geometric
 offspring distribution~$p_k = \mu^k \,(1 - \mu)$ which has mean~$m = \mu \,(1 - \mu)^{-1}$.
 For such a tree, the corresponding discrete-time branching process~$Z_j$ has expected value~$E Z_j = m^j$ therefore
 $$ P \,(Z_{L/M} > 0) \ = \ P \,(Z_{L/M} \geq 1) \ \leq \ E (Z_{L/M}) \ = \ m^{L/M}. $$
 The probability that a subtree has depth at least~$L/M$ is thus
 $$ \begin{array}{l}
    \leq \ \exp \,(- M (a + b) \,NT/2) + 4M (a + b) \,NT \,m^{L/M} \vspace*{4pt} \\
    \hspace*{25pt} \leq \ N^{- (a + b) \,MN} + 8M (a + b) \,\ln (N) \,N^{1 - (\alpha_2 / M) \,|\ln m|} \ \leq \ N^{-1} \end{array} $$
 for all~$N$ large and~$\ep > 0$ small, since~$\lim_{\ep \to 0} \,m = \lim_{\ep \to 0} \,(a + b) \,\ep = 0$.
\end{proof}


\section{Block construction for spread and proof of Theorem~\ref{th:spread}}
\label{sec:spread}

\indent This section is devoted to the proof of Theorem~\ref{th:spread}, which follows the same three-step process as the proofs of the previous two theorems.
 More precisely,
\begin{enumerate}
 \item we first study the mean-field equations~\eqref{eq:mean-field} starting with a single fully occupied patch and all the other patches empty, \vspace*{4pt}
 \item we then use the convergence in distribution of the~$N$-dual to the limiting dual in order to show that, at least in a bounded space-time box, the stochastic process behaves almost like its deterministic counterpart when~$N$ is large, \vspace*{4pt}
 \item we finally use a block construction to deduce that the probability of survival of the stochastic process starting with a single fully occupied patch approaches one for~$N$ large.
\end{enumerate}
 Motivated by the monotonicity result stated in Corollary~\ref{cor:mt}, we first study the system where two adjacent patches can interact while all the other patches remain empty at all times.
 This system is used to understand a single time step in the block construction.
 That is, we let
 $$ u = u_x \quad \hbox{and} \quad v = u_y \quad \hbox{where} \quad x, y \in \Z \ \ \hbox{with} \ \ |x - y| = 1 $$
 and study the following system of coupled differential equations:
\begin{equation}
\label{eq:two-patches}
  \begin{array}{rclclcl}
    u' & = & F (u, v) & := & (a u^2 + (b/2) \,v^2)(1 - u) - u  \vspace*{4pt} \\
    v' & = & F (v, u) & := & (a v^2 + (b/2) \,u^2)(1 - v) - v. \end{array}
\end{equation}
 Assuming that~$r = a + b/2 > 4$ and letting
 $$ u_- \ := \ 1/2 - (1/4 - 1/r)^{1/2} \quad \hbox{and} \quad u_+ \ := \ 1/2 + (1/4 - 1/r)^{1/2} $$
 some basic algebra implies that
\begin{equation}
\label{eq:fixed-points}
  F (u, u) \ = \ - ru \,(u - u_-)(u - u_+).
\end{equation}
 To prove spread of the mean-field model, we proceed in two steps. \vspace*{4pt} \\
\noindent {\bf Step 1} -- Starting with patch~$x$ fully occupied and patch~$y$ empty, the density in both patches converges to the locally stable fixed point~$u_+$. \vspace*{4pt} \\
\noindent {\bf Step 2} -- Assuming that the density in patch~$x$ is at least~$u_+$ at all times, the density in patch~$y$ converges to~$u_+$ as well regardless of its initial value. \vspace*{4pt} \\
 These two steps are established in the next two lemmas, respectively.
\begin{lemma} --
\label{lem:first-step}
 Let~$r > 4$ and~$b > 2$. Then,
 $$ \begin{array}{l} u (0) = 1 \ \ \hbox{and} \ \ v (0) = 0 \quad \hbox{imply that} \quad \lim_{t \to \infty} \,u (t) = \lim_{t \to \infty} \,v (t) = u_+. \end{array} $$
\end{lemma}
\begin{proof}
 To begin with, we observe that, whenever~$u + v = 1$ and~$b > 2$,
\begin{equation}
\label{eq:first-step-1}
  \begin{array}{rcl}
    (u + v)' & = & (a u^2 + (b/2) \,v^2) \,v - u + (a v^2 + (b/2) \,u^2) \,u - v \vspace*{4pt} \\
                 & = & (a - 3b/2) \,u \,v + b/2 - 1 \ \geq \ \min \,((1/8)(2a + b) - 1, b/2 - 1) \vspace*{4pt} \\ & = & \min \,(r/4 - 1, b/2 - 1) \ > \ 0. \end{array}
\end{equation}
 In addition, using~\eqref{eq:fixed-points} and that~$v \mapsto F (u, v)$ is nondecreasing, we get
\begin{equation}
\label{eq:first-step-2}
  \begin{array}{rcl}
    u' \ = \ F (u, v) \ > \ F (u, u) \ > \ 0 & \hbox{when} & u_- < u < \min \,(v, u_+) \vspace*{2pt} \\
    u' \ = \ F (u, v) \ < \ F (u, u) \ < \ 0 & \hbox{when} & u > v > u_+ \vspace*{2pt} \\
    u' \ = \ F (u, v) \ < \ F (u, u) \ < \ 0 & \hbox{when} & u + v > 1 \ \hbox{and} \ u > u_+ \ \hbox{and} \ v < u_-. \end{array}
\end{equation}
 These three inequalities and their analogs obtained by switching the roles of~$u$ and~$v$ are summarized in the picture of Figure~\ref{fig:two-patch}.
 Combining~\eqref{eq:first-step-1} and the last statement in~\eqref{eq:first-step-2} implies that, under the assumptions of the lemma, there exists a time~$t$ finite such that
 $$ (u + v)(t) > 1 \quad \hbox{and} \quad \min \,(u (t), v (t)) > u_-. $$
 This, together with~\eqref{eq:fixed-points} and the first two statements in~\eqref{eq:first-step-2}, implies the lemma.
\end{proof}
\begin{figure}[t]
\centering
\scalebox{0.40}{\input{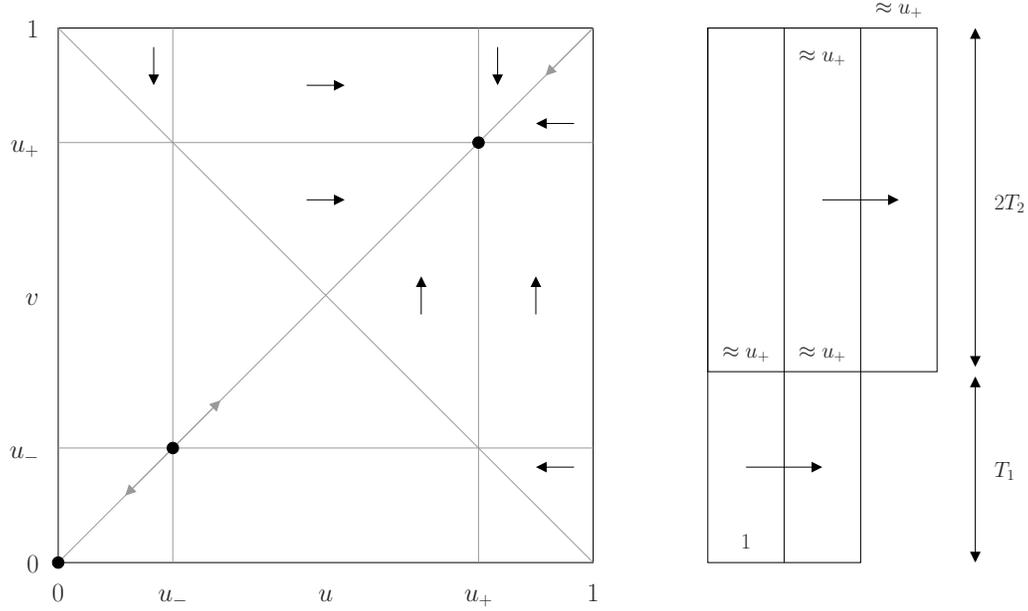}}
\caption{\upshape{Summary of the inequalities in~\eqref{eq:first-step-2} showing the sign of the derivatives on the left and schematic picture of the two steps of the block construction used to prove Theorem~\ref{th:spread} on the right.}}
\label{fig:two-patch}
\end{figure}
\begin{lemma} --
\label{lem:second-step}
 Let~$r > 4$ and~$b > 2$.
 Then, there are~$\ep > 0$ and~$T < \infty$ such that
 $$ \begin{array}{l} u (t) \geq u_+ - \ep \ \ \hbox{for all} \ \ t \leq T \quad \hbox{implies that} \quad \lim_{t \to \infty} \,u (t) = \lim_{t \to \infty} \,v (t) = u_+. \end{array} $$
\end{lemma}
\begin{proof}
 We prove the result when~$u (t) = u_+ - \ep$ for all~$t \leq T$ and~$a \leq 8$.
 In view of the monotonicity of the system, this will imply that the lemma holds in the general case.
 Under these specific assumptions, the second equation in~\eqref{eq:two-patches} becomes
\begin{equation}
\label{eq:second-step-1}
 \begin{array}{rcl}
    v' & = & (a v^2 + (b/2)(u_+ - \ep)^2)(1 - v) - v \vspace*{4pt} \\
         & = & (a v^2 + (b/2) \,u_+^2)(1 - v) - v + (b/2)(\ep^2 - 2 \ep \,u_+)(1 - v) \vspace*{4pt} \\
         & \geq & (u_+ - v) \,Q (v) - b \ep
 \end{array}
\end{equation}
 where~$Q (X) := a X^2 + a \,(u_+ - 1) \,X + (b/2) \,u_+$.
 For all~$X \in (0, 1)$ and~$a \leq 8$,
\begin{equation}
\label{eq:second-step-2}
  \begin{array}{rcl}
    Q (X) & \geq & Q (- a \,(u_+ - 1) / 2a) \ = \ Q ((1/2) \,u_-) \vspace*{4pt} \\
             & = & (a/4) \,u_-^2 + (a/2)(u_+ - 1) \,u_- + (b/2) \,u_+ \vspace*{4pt} \\
             & = & - \ (a/4) \,u_-^2 + (b/2) \,u_+ \ > \ b/4 - a/16 \ \geq \ b/4 - 1/2 \ > \ 0. \end{array}
\end{equation}
 Combining~\eqref{eq:second-step-1}--\eqref{eq:second-step-2}, we deduce that, for all~$\ep > 0$ small,
 $$ \begin{array}{l} v \leq 1/2 \quad \hbox{implies that} \quad v' \ = \ (u_+ - 1/2)(b/4 - 1/2) - b \ep \ > \ 0. \end{array} $$
 In particular, for all~$\ep > 0$ small, there exists~$T < \infty$ such that
 $$ \begin{array}{l} u (t) = u_+ - \ep \ \ \hbox{for all} \ \ t \leq T \quad \hbox{implies that} \quad u (T) > 1/2 \ \ \hbox{and} \ \ v (T) > 1/2. \end{array} $$
 This, together with the direction of the arrows in Figure~\ref{fig:two-patch}, shows that
 $$ \begin{array}{l} \lim_{t \to \infty} \,u (t) = \lim_{t \to \infty} \,v (t) = u_+ \end{array} $$
 This completes the proof.
\end{proof} \\ \\
 With the previous two lemmas in hands, we can now fix the appropriate time scale for the block construction.
 Let~$\ep > 0$ be small.
 Then, according to Lemma~\ref{lem:first-step}, there exists a time~$T_1 < \infty$, fixed from now on, such that we have
\begin{equation}
\label{eq:step-1}
  v (t) \geq u_+ - \ep / 2 \ \ \hbox{for all} \ \ t \geq T_1 \quad \hbox{when} \quad u (0) = 1.
\end{equation}
 Also, by Lemma~\ref{lem:second-step}, there exists~$T_2 < \infty$, fixed from now on, such that
\begin{equation}
\label{eq:step-2}
  v (t) \geq u_+ - \ep / 2 \ \ \hbox{for all} \ \ t \geq T_2 \quad \hbox{when} \quad u (t) \geq u_+ - \ep \ \ \hbox{for all} \ \ t \leq 2 T_2.
\end{equation}
 For these deterministic times~$T_1$ and~$T_2$, we have the following lemmas which can be seen as the analogs of the previous two lemmas but for the stochastic process.
\begin{lemma} --
\label{lem:first-invasion}
 Let~$r = a + b/2 > 4$ and~$b > 2$. Then,
 $$ \begin{array}{l} \lim_{N \to \infty} \,P_o \,(\xi_t (1) \leq (u_+ - \ep) N \ \hbox{for some} \ t \in (T_1, T_1 + 2 T_2)) \ = \ 0. \end{array} $$
\end{lemma}
\begin{proof}
 Using Lemma~\ref{lem:collision-estimate}, Theorem~\ref{th:occupation-density} and~\eqref{eq:step-1}, we get
 $$ \begin{array}{l}
      P_o \,(\xi_t (1) \leq (u_+ - \ep) N \ \hbox{for some} \ t \in (T_1, T_1 + 2 T_2)) \vspace*{4pt} \\ \hspace*{40pt} = \
      P_o \,(u_1^N (t) \leq u_+ - \ep \ \hbox{for some} \ t \in (T_1, T_1 + 2 T_2)) \vspace*{4pt} \\ \hspace*{40pt} \leq \
      P_o \,(|u_1^N (t) - v (t)| \geq \ep / 2 \ \hbox{for some} \ t \in (T_1, T_1 + 2 T_2)) \vspace*{4pt} \\ \hspace*{40pt} \leq \
      2 \,(\ep / 2)^{-2} \,P \,(\tau^N \leq T_1 + 2 T_2) \ \leq \
      8 \,\ep^{-2} \,(2 \,e^{2(a + b)(T_1 + 2 T_2)} + 1) \,N^{-1/3}. \end{array} $$
 This completes the proof.
\end{proof}
\begin{lemma} --
\label{lem:next-invasion}
 Let~$r = a + b/2 > 4$ and~$b > 2$. Then,
 $$ \begin{array}{l}
    \lim_{N \to \infty} \,P \,(\xi_t (1) \leq (u_+ - \ep) N \ \hbox{for some} \ t \in (T_2, 3 T_2) \ | \vspace*{4pt} \\ \hspace*{100pt} \xi_t (0) > (u_+ - \ep) N \ \hbox{for all} \ t \in (0, 2 T_2)) \ = \ 0. \end{array} $$
\end{lemma}
\begin{proof}
 Following the proof of Lemma~\ref{lem:first-invasion} but using~\eqref{eq:step-2} instead of~\eqref{eq:step-1}, we easily prove that the conditional probability to be estimated is at most
 $$ 2 \,(\ep / 2)^{-2} \,P \,(\tau^N \leq 3 T_2) \ \leq \ 8 \,\ep^{-2} \,(2 \,e^{6 (a + b) \,T_2} + 1) \,N^{-1/3} $$
 which goes to zero as~$N \to \infty$.
\end{proof} \\ \\
 Using the previous two lemmas, we can now prove the theorem. \\ \\
\begin{demo}{Theorem~\ref{th:spread}} --
 Declare site~$(z, n) \in H$ to be good whenever
 $$ \xi_t (z) > (u_+ - \ep_1) N \quad \hbox{for all} \quad t \in (T_1 + n T_2, T_1 + (n + 2) \,T_2) $$
 and let~$X_n$ be the set of good sites at level $n$, i.e.,
 $$ X_n \ := \ \{z \in \Z : (z, n) \in H \ \hbox{is good} \}. $$
 Lemma~\ref{lem:first-invasion} and obvious symmetry imply that
\begin{equation}
\label{eq:spread-2}
 \begin{array}{l}
  P_o \,(\{-1, +1 \} \not \subset X_0) \vspace*{4pt} \\ \hspace*{25pt} = \
  P_o \,(\min_{z = -1, 1} \,\xi_t (z) \leq (u_+ - \ep) N \ \hbox{for some} \ t \in (T_1, T_1 + 2 T_2)) \vspace*{4pt} \\ \hspace*{25pt} \leq \
  2 \ P_o \,(\xi_t (1) \leq (u_+ - \ep) N \ \hbox{for some} \ t \in (T_1, T_1 + 2 T_2)) \to 0 \ \ \hbox{as} \ \ N \to \infty \end{array}
\end{equation}
 while, according to Lemma~\ref{lem:next-invasion} and again symmetry,
\begin{equation}
\label{eq:spread-3}
 \begin{array}{l}
  P \,(z \pm 1 \notin X_{n + 1} \ | \ z \in X_n) \vspace*{4pt} \\ \hspace*{25pt} = \
  P \,(\xi_t (z \pm 1) \leq (u_+ - \ep) N \ \hbox{for some} \ t \in (T_1 + (n + 1) \,T_2, T_1 + (n + 3) \,T_2) \ | \vspace*{4pt} \\ \hspace*{75pt}
       \xi_t (z) > (u_+ - \ep) N \ \hbox{for all} \ t \in (T_1 + n T_2, T_1 + (n + 2) \,T_2)) \vspace*{4pt} \\ \hspace*{25pt} = \
  P \,(\xi_t (1) \leq (u_+ - \ep) N \ \hbox{for some} \ t \in (T_2, 3 T_2) \ | \vspace*{4pt} \\ \hspace*{100pt}
       \xi_t (0) > (u_+ - \ep) N \ \hbox{for all} \ t \in (0, 2 T_2)) \to 0 \ \ \hbox{as} \ \ N \to \infty. \end{array}
\end{equation}
 Letting~$W_n$ be the set of wet sites at level~$n$ in an oriented site percolation process in which sites are open with probability~$1 - \gamma$, the limit in~\eqref{eq:spread-3} and Theorem~\ref{th:st-flour-3} imply that, for all~$\gamma > 0$,
 there exist~$N$ large and a coupling of the two processes such that
 $$ W_n \,\subset \,X_n \ \ \hbox{for all} \ n \in \Z_+ \quad \hbox{whenever} \quad W_0 \,\subset \,X_0. $$
 This, together with the limit in~\eqref{eq:spread-2}, implies that
 $$ \begin{array}{l}
      P_o \,(\xi_t = \0 \ \hbox{for some} \ t > 0) \vspace*{4pt} \\ \hspace*{25pt} \leq \
      P \,(X_n = \varnothing \ \hbox{for some} \ n \ | \ X_0 = \{-1, +1 \}) \ + \ P_o \,(\{-1, +1 \} \not \subset X_0) \vspace*{4pt} \\ \hspace*{25pt} \leq \
      P \,(W_n = \varnothing \ \hbox{for some} \ n \ | \ W_0 = \{-1, +1 \}) \ + \ P_o \,(\{-1, +1 \} \not \subset W_0) \end{array} $$
 which, according to Theorem~\ref{th:st-flour-1}, goes to zero as~$N \to \infty$.
\end{demo}


\section{Long range dispersal and proof of Theorem \ref{th:range}}
\label{sec:range-extinction}

\indent This last section is devoted to the proof of Theorem~\ref{th:range} which shows that, for any choice of the inner and outer birth rates~$a$ and~$b$ and patch size~$N$, the probability that the process
 starting with a single occupied patch survives tends to zero as the dispersal range tends to infinity.
 This together with Theorem~\ref{th:spread} suggests that dispersal promotes extinction of processes with sexual reproduction, whereas dispersal is known to promote survival of the basic contact process
 with no sexual reproduction.
 The proof relies on the following two key ingredients:
\begin{list}{\labelitemi}{\leftmargin=1.75em}
 \item[1.] In the absence of migrations: $b = 0$, the process starting with a single fully occupied patch goes extinct almost surely in a finite time.
  This directly follows from the fact that, in this case, the process converges to its unique absorbing state because the state space is finite. \vspace*{4pt}
 \item[2.] Calling a {\bf collision} the event that two offspring produced at the source patch $x$ are sent to the same target patch $y \neq x$, in the presence of migrations: $b \neq 0$, but in the absence
  of collisions, offspring sent outside the source patch cannot reproduce due to the birth mechanism.
  In particular, in view also of the previous point, the process dies out.
\end{list}
 It follows that, to find an upper bound for the survival probability, it suffices to find an upper bound for the probability of a collision since we have
\begin{equation}
\label{eq:range-1}
  P_o \,(\xi_t \neq \0 \ \hbox{for all} \ t > 0) \ \leq \ P_o \,(\hbox{collision}).
\end{equation}
 In addition, since the probability of a collision is related to the number of individuals produced at patch~$x$ and sent outside the patch which, in turn, is related to the number of individuals at the source
 patch in the absence of migrations, to find an upper bound for the probability of a collision, the first step is to find an upper bound for the time spent in state~$j$
 defined as
 $$ \begin{array}{l} \tau_j \ := \int_{\R_+} P \,(X_t = j \,| \,X_0 = N) \,dt \quad \hbox{for} \quad j = 1, 2, \ldots, N. \end{array} $$
 More precisely, we have the following upper bound.
\begin{lemma} --
\label{lem:range-extinction}
 Let $b = 0$.
 Then, $\sum_{j = 1, 2 \ldots, N} \ j \,\tau_j \,\leq \,\sum_{j = 1, 2, \ldots, N} \sum_{i = 0, 1, \ldots, j} \ (a/4)^i$.
\end{lemma}
\begin{proof}
 To simplify some tedious calculations and find the upper bound in the statement of the lemma, we first observe that the number of individuals at patch $x$ is dominated by the number of particles in a certain
 simple birth and death process truncated at state $N$.
 More precisely, we note that the rate of transition $j \to j + 1$ is bounded by
 $$ a \,j \,(j - 1) \,N^{-1} (N - 1)^{-1}(N - j) \ \leq \ a \,j^2 \,N^{-2} (N - j) \ \leq \ (a/4) \,j. $$
 In particular, standard coupling arguments imply that
\begin{equation}
\label{eq:range-extinction-1}
  P \,(X_t \geq i \ | \ X_0 = N) \ \leq \ P \,(Z_t \geq i \ | \ Z_0 = N) \quad \hbox{for all} \quad i = 1, 2, \ldots, N
\end{equation}
 where $Z_t$ is the continuous-time Markov chain with transitions
 $$ \begin{array}{rclcll}
     j & \to & j + 1 & \ \ \hbox{at rate} \ \ & \beta_j := (a/4) \,j & \ \ \hbox{for} \ j = 0, 1, \ldots, N - 1 \vspace*{2pt} \\
     j & \to & j - 1 & \ \ \hbox{at rate} \ \ &   \mu_j :=         j & \ \ \hbox{for} \ j = 0, 1, \ldots, N. \end{array} $$
 Therefore, letting $\sigma_j$ denote the amount of time the process $Z_t$ spends in state $j$, which can be seen as the analog of the amount of time $\tau_j$, inequality \eqref{eq:range-extinction-1} implies that
\begin{equation}
\label{eq:range-extinction-2}
 \begin{array}{rcl}
   \sum_{j = 1, 2, \ldots, N} \,j \,\tau_j & = &
   \int_{\R_+} \sum_{j = 1, 2, \ldots, N} \,\sum_{i = 1, 2, \ldots, j} \,P \,(X_t = j \ | \,X_0 = N) \,dt \vspace*{4pt} \\ & = &
   \int_{\R_+} \sum_{i = 1, 2, \ldots, N} \,\sum_{j = i, \ldots, N} \,P \,(X_t = j \ | \,X_0 = N) \,dt \vspace*{4pt} \\ & = &
   \int_{\R_+} \sum_{i = 1, 2, \ldots, N} \,P \,(X_t \geq i \ | \,X_0 = N) \,dt \vspace*{4pt} \\ & \leq &
   \int_{\R_+} \sum_{i = 1, 2, \ldots, N} \ P \,(Z_t \geq i \ | \,Z_0 = N) \,dt \ = \ \sum_{j = 1, 2, \ldots, N} \,j \,\sigma_j. \end{array}
\end{equation}
 Now, to find an upper bound for the occupation times~$\sigma_j$, we first let $v_j$ denote the expected number of visits of the process~$Z_t$ in state~$j$, which gives the recursive relationship
\begin{equation}
\label{eq:range-extinction-3}
 \begin{array}{rcl}
   v_j & = & \mu_{j + 1} \,(\beta_{j + 1} + \mu_{j + 1})^{-1} \,v_{j + 1} + \beta_{j - 1} \,(\beta_{j - 1} + \mu_{j - 1})^{-1} \,v_{j - 1} \vspace*{4pt} \\
       & = & 4 \,(a + 4)^{-1} \,v_{j + 1} + a \,(a + 4)^{-1} \,v_{j - 1} \end{array}
\end{equation}
 for $j = 1, 2, \ldots, N - 1$, with boundary conditions
\begin{equation}
\label{eq:range-extinction-4}
   v_0 \ = \ \mu_1 \,(\beta_1 + \mu_1)^{-1} \,v_1 \quad \hbox{and} \quad
   v_N \ = \ 1 + \beta_{N - 1} \,(\beta_{N - 1} + \mu_{N - 1})^{-1} \,v_{N - 1}.
\end{equation}
 Note that the extra one in the expression of $v_N$ comes from the fact that $Z_0 = N$.
 We observe that the recursive relationship \eqref{eq:range-extinction-3} can be re-written as
 $$ v_j \ = \ (1 + a/4) \,v_{j - 1} - (a/4) \,v_{j - 2} \quad \hbox{for} \quad j = 2, 3, \ldots, N $$
 which has characteristic polynomial
 $$ X^2 - (1 + a/4) X + a/4 \ = \ (X - 1)(X - a/4). $$
 Using in addition that $v_0 = 1$ since state zero is absorbing, we get
\begin{equation}
\label{eq:range-extinction-5}
  v_1 \ = \ 1 + (a/4) \quad \hbox{and} \quad v_j \ = \ c_{22} + c_{23} \,(a/4)^j \ = \ 1 + (a/4) + \cdots + (a/4)^j
\end{equation}
 from which it follows that
\begin{equation}
\label{eq:range-extinction-6}
 \begin{array}{rcl}
   \sigma_j & = & v_j \,(\beta_j + \mu_j)^{-1} \ = \ j^{-1} \,(1 + a/4)^{-1} \ \sum_{i = 0, 1, \ldots, j} \,(a/4)^i \vspace*{4pt} \\
            & \leq & j^{-1} \ \sum_{i = 0, 1, \ldots, j} \,(a/4)^i \end{array}
\end{equation}
 for $j = 1, 2, \ldots, N - 1$.
 Using \eqref{eq:range-extinction-4}--\eqref{eq:range-extinction-5}, we also deduce
\begin{equation}
\label{eq:range-extinction-7}
 \begin{array}{rcl}
    \sigma_N & = & \mu_N^{-1} \,(1 + \beta_{N - 1} \,(\beta_{N - 1} + \mu_{N - 1})^{-1} \,v_{N - 1}) \vspace*{4pt} \\
             & \leq & N^{-1} \,(1 + (a/4)(1 + a/4)^{-1} \,\sum_{i = 0, 1, \ldots, N - 1} \,(a/4)^i) \vspace*{4pt} \\
             & \leq & N^{-1} \,(1 + (a/4) \ \sum_{i = 0, 1, \ldots, N - 1} \,(a/4)^i) \ = \ N^{-1} \,\sum_{i = 0, 1, \ldots, N} \,(a/4)^i. \end{array}
\end{equation}
 Finally, combining \eqref{eq:range-extinction-2} with \eqref{eq:range-extinction-6}--\eqref{eq:range-extinction-7}, we get
 $$ \begin{array}{l}
    \sum_{j = 1, 2, \ldots, N} \,j \,\tau_j \ \leq \
    \sum_{j = 1, 2, \ldots, N} \,j \,\sigma_j \ \leq \
    \sum_{j = 1, 2, \ldots, N} \,\sum_{i = 0, 1, \ldots, j} \,(a/4)^i \end{array} $$
 which completes the proof.
\end{proof} \\ \\
 Using Lemma \ref{lem:range-extinction}, we can now deduce upper bounds for the probability of a collision which, together with the inequality in~\eqref{eq:range-1}, also give the theorem. \\ \\
\begin{demo}{Theorem~\ref{th:range}} --
 Let $X$ denote the number of individuals born in patch $x$ and then sent outside the patch before the patch goes extinct.
 The idea is to use
\begin{equation}
\label{eq:range-collision-1}
 \begin{array}{rcl}
   P_o \,(\hbox{collision}) & = &
   P   \,(\hbox{collision} \,| \,X \leq M^{1/3}) \,P_o \,(X \leq M^{1/3}) \vspace*{4pt} \\ && \hspace*{25pt} + \
   P   \,(\hbox{collision} \,| \,X > M^{1/3}) \,P_o \,(X > M^{1/3}) \vspace*{4pt} \\ & \leq &
   P   \,(\hbox{collision} \,| \,X \leq M^{1/3}) + P_o \,(X > M^{1/3}). \end{array}
\end{equation}
 To estimate the first term in \eqref{eq:range-collision-1}, we observe that, since offspring sent outside the patch land on a patch chosen
 uniformly at random from a set of $2M$ patches, we have
\begin{equation}
\label{eq:range-collision-2}
  \begin{array}{l}
   P \,(\hbox{collision} \,| \,X \leq M^{1/3}) \ \leq \ 1 - \prod_{j = 0, 1, \ldots, M^{1/3} - 1} \,(1 - j / 2M) \vspace*{4pt} \\ \hspace*{25pt} \leq \
   1 - (1 - M^{1/3} / 2M)^{M^{1/3}} \ \leq \ 1 - \exp (- M^{2/3} / 2M) \ \leq \ (1/2) \,M^{-1/3} \end{array}
\end{equation}
 for all $M$ large.
 For the second term, we first use Lemma~\ref{lem:range-extinction} to get
 $$ \begin{array}{rcl}
     E_x (X) & = &    b \,N \ \sum_{j = 2, 3 \ldots, N} \,j \,(j - 1) \,N^{-1} \,(N - 1)^{-1} \ \tau_j \vspace*{4pt} \\
             & \leq & b \ \sum_{j = 2, 3, \ldots, N} \,j \,\tau_j \ \leq \ b \ \sum_{j = 1, 2, \ldots, N} \,\sum_{i = 0, 1, \ldots, j} \,(a/4)^i. \end{array} $$
 Looking at the different values of~$a$, we deduce that~$E_x (X)$ is bounded by
 $$ \begin{array}{rcl}
      b \,(1 - a/4)^{-1} \,\sum_{j = 1, 2, \ldots, N} \,(1 - (a/4)^{j + 1}) \leq b \,N (1 - a/4)^{-1}         & \hbox{when} & a < 4  \vspace*{4pt} \\
      b \,\sum_{j = 1, 2, \ldots, N} \,(j + 1) \leq (b/2)(N + 2)^2                                            & \hbox{when} & a = 4  \vspace*{4pt} \\
      b \,(a/4 - 1)^{-1} \,\sum_{j = 1, 2, \ldots, N} \,(a/4)^{j + 1} \leq b \,(a/4 - 1)^{-2} \,(a/4)^{N + 2} & \hbox{when} & a > 4. \end{array} $$
 This and Markov's inequality~$P_o \,(X > M^{1/3}) \leq M^{- 1/3} \,E_x (X)$ give
\begin{equation}
\label{eq:range-collision-3}
  \begin{array}{rclcl}
    P_o \,(X > M^{1/3}) & \leq & M^{-1/3} \,b \,N (1 - a/4)^{-1}               & \hbox{when} & a < 4  \vspace*{4pt} \\
                        & \leq & M^{-1/3} \,(b/2)(N + 2)^2                     & \hbox{when} & a = 4  \vspace*{4pt} \\
                        & \leq & M^{-1/3} \,b \,(a/4 - 1)^{-2} \,(a/4)^{N + 2} & \hbox{when} & a = 4. \end{array}
\end{equation}
 Theorem~\ref{th:range} directly follows from~\eqref{eq:range-collision-1}--\eqref{eq:range-collision-3} also using~\eqref{eq:range-1}.
\end{demo}


\end{document}